\documentclass[aap,citesort,MSNbibl,seceqn,dvips]{arximspdf}
\usepackage{accents}
\usepackage{graphicx}

\doi{10.1214/10-AAP733}
\volume{21}
\issue{5}
\pubyear{2011}
\firstpage{1749}
\lastpage{1794}

\makeatletter
\def\bptnote#1{}
\newcommand{\notag}{\nonumber}
\renewcommand{\underline}[1]{\underaccent{\bar}#1}
\renewcommand{\overline}[1]{\bar{#1}}

\renewcommand{\P}{\mathbb{P}}
\newcommand{\E}{\mathbb{E}}
\newcommand{\N}{\mathbb{N}}
\newcommand{\R}{\mathbb{R}}

\newcommand{\F}{\mathcal{F}}
\newcommand{\G}{\mathcal{G}}
\newcommand{\PP}{\mathcal{P}}
\newcommand{\LL}{\mathcal{L}}
\newcommand{\TT}{\mathcal{T}}
\renewcommand{\epsilon}{\varepsilon}

\def\op{\overline p}
\newproclaim{definition}{Definition}
\newproclaim{rem}{Remark}
\newtheorem{theorem}{Theorem}
\newtheorem{proposition}{Proposition}
\newtheorem{lemma}{Lemma}
\makeatother

\begin{document}
\begin{frontmatter}

\title{Traveling waves and homogeneous fragmentation}
\runtitle{Traveling waves and homogeneous fragmentation}

\begin{aug}
\author[A]{\fnms{J.} \snm{Berestycki}\ead[label=e1]{julien.berestycki@upmc.fr}\thanksref{a1}},
\author[B]{\fnms{S. C.} \snm{Harris}\ead[label=e2]{s.c.harris@bath.ac.uk}}
\and
\author[B]{\fnms{A. E.} \snm{Kyprianou}\corref{}\ead[label=e3]{a.kyprianou@bath.ac.uk}}
\runauthor{J. Berestycki, S. C. Harris and A. E. Kyprianou}
\affiliation{Universit\'e Paris VI, University of Bath  and University of Bath}
\address[A]{J. Berestycki\\
Laboratoire de Probabilit\'es\\
\quad  et Mod\`eles Al\'eatoires\\
Universit\'e Paris VI\\
 Case courrier 188\\
  4, Place Jussieu\\
75252 Paris Cedex 05\\
France\\
\printead{e1}} %adresu isvedimo komanda gale!
\address[B]{S. C. Harris\\
 A. E. Kyprianou\\
 Department of Mathematical Sciences\\
 University of Bath\\
  Claverton Down\\
   Bath BA2 7AY\\
   UK\\
\printead{e2}\\
\phantom{\textsc{E-mail}: }\printead*{e3}}
\end{aug}
\thankstext{a1}{Supported in part by a Royal Society
International Incoming Visitors grant for which all three authors are grateful.}

% HISTORY:
\received{\smonth{11} \syear{2009}}
\revised{\smonth{9} \syear{2010}}

% ABSTRACT
%
\begin{abstract}
We formulate the notion of the classical Fisher--Kolmogorov--Petrovskii--Piscounov (FKPP)
reaction diffusion equation associated with a homogeneous conservative fragmentation process and study
its traveling waves. Specifically, we establish existence, uniqueness and asymptotics. In the spirit of
classical works such as McKean
[\textit{Comm. Pure Appl. Math}. \textbf{28} (1975) 323--331] and
[\textit{Comm. Pure Appl. Math.} \textbf{29} (1976) 553--554], Neveu [In \textit{Seminar on Stochastic Processes}
(1988) 223--242 Birkh\"auser]
and Chauvin [\textit{Ann. Probab}. \textbf{19} (1991) 1195--1205],
our analysis exposes the relation
between traveling waves and certain additive and multiplicative martingales via laws of large numbers  which have
been previously studied in the context of Crump--Mode--Jagers (CMJ) processes by
Nerman
[\textit{Z. Wahrsch. Verw. Gebiete} \textbf{57} (1981) 365--395] and in the context of
fragmentation processes by Bertoin and Martinez [\textit{Adv. in Appl. Probab.} \textbf{37} (2005) 553--570]
and Harris, Knobloch and Kyprianou
[\textit{Ann. Inst. H. Poincar\'e Probab. Statist.} \textbf{46}
  (2010)  119--134].
The conclusions and methodology presented here appeal to a number of concepts coming from the theory of branching
random walks and branching Brownian motion (cf. Harris [\textit{Proc. Roy. Soc. Edinburgh Sect.~A}
 \textbf{129} (1999) 503--517]
and Biggins and Kyprianou [\textit{Electr. J. Probab.} \textbf{10} (2005) 609--631]) showing their mathematical
robustness even within the context of fragmentation theory. %We also make frequent use of renewal theory for
%L\'evy processes with no negative jumps.
\end{abstract}

% KEYWORDS
%

\end{frontmatter}

%s1 ###
\section{Introduction and main results}\label{intro}

%s1.1 ###
\subsection{Homogeneous fragmentations and branching random walks}\label{sec1.1}

Fragmentation is a natural phenomena that occurs in a wide range of
contexts and at all scales. The stochastic models used to describe this
type of process have attracted a lot of attention lately and form a
fascinating class of mathematical objects in their own right, which are
deeply connected to branching processes, continuum random trees and
branching random walks. A good introduction to the study of
fragmentation (and coalescence) is \cite{Be3} which also contains many
further references.

In the present work, we intend to explore and make use of the
connection between random fragmentation processes, branching random
walk (BRW) and branching Brownian motion (BBM). More precisely we
define the notion of the fragmentation
Fisher--Kolmogorov--Petrovskii--Piscounov (FKPP) equation and study the
solutions of this equation.

Let us start by explaining the connection between fragmentation and BRW
in the simple framework of \textit{finite-activity} conservative
fragmentations. In this context everything can be defined and
constructed \textit{by hand}. More general constructions will follow. Let
$\nu(\cdot)$ be a finite measure on $\nabla_1 =\{ s_1\ge s_2 \ge
\cdots\ge0, \sum_i s_i =1 \}$ with total mass $\nu(\nabla
_1)=\gamma.$ The homogeneous mass-fragmentation process with
dislocation measure $\nu$ is a $\nabla_1$-valued Markov process $X \dvtx
= (X(t),t\ge0) $, where $X(t) = (X_1(t),X_2(t), \ldots )$, which
evolves as follows: $X(0)=(1,0,\ldots );$ this initial fragment then
waits an exponential time $T_1$ with parameter $\gamma$ after which it
splits according to the distribution $X(T_1) \sim\gamma^{-1} \nu
(\cdot).$ Each of these pieces then starts to fragment independently
of the others with the same law as the original object. That is, each
fragment $X_i(T_1)$ waits an independent exponential-$\gamma$ time
after which it splits into $( X_i(T_1)s_1, X_i(T_1) s_2,\ldots )$ where
$s =(s_1,s_2,\ldots )\sim\gamma^{-1} \nu(\cdot)$ and so on. When a
fragment splits, we need to relabel all fragments since their relative
ranks have changed.

This process can be seen as a continuous-time BRW. More precisely, if
we let $ Z(t) = (-\log X_1(t), -\log X_2(t), \ldots )$, then $Z(t)$
evolves according to the following dynamic. Suppose $Z(t) = (z_1, z_2,
\ldots ),$ then each individual in the population behaves
independently, waits an independent and exponentially distributed length of time with
parameter $\gamma$ and then branches into offspring which are situated
at distances $(-\log s_1, -\log s_2, \ldots )$ relative to their
parent's position where, as before, $s\in\nabla_1$ has distribution
$\gamma^{-1}\nu(\cdot).$
Figure~\ref{F:frag rw} %jb prob ref, write "1" ?
shows an example where $\nu(\cdot) = \delta_{\{1/3,1/3,1/3,0,\ldots
\}} + \delta_{\{1/2,1/2,0,\ldots \}}.$
%Thus we can draw form a large literature of results concerning
%branching random walks to study processes like $F_t.$

%f1 ###
\begin{figure}%[b]

\includegraphics{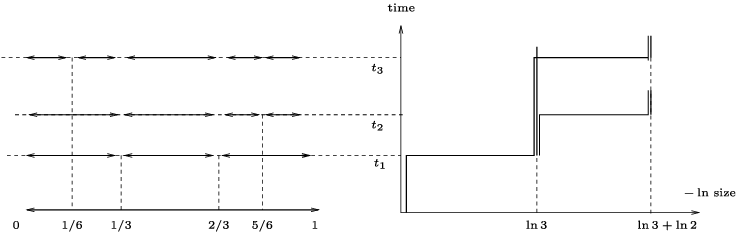}

\caption{The initial fragment $(0,1)$ splits into three equal parts
$(0,1/3), (1/3 , 2/3), (2/3,1)$ at time $t_1$. Then at $t_2, (2/3,1)$
splits into two halves and at time $t_3$ the same thing happens to $(0,1/3)$.}
\label{F:frag rw}
\end{figure}

This is not, however, the whole story. In particular, it is possible to
define homogenous fragmentation processes, also denoted $X$, such that
the dislocation measure $\nu$ is infinite (but sigma-finite). In this
case, fragmentation happens continuously in the sense that there is no
first splitting event and along any branch of the fragmentation tree
the branching points are dense in any finite interval of time.
Moreover, it is still the case that the Markov and fragmentation
properties hold. In this more general setting the latter two properties
are described more succinctly as follows. Given that $X(t) =
(s_1,s_2,\ldots )\in\nabla_1$, where $t\geq0$, then for $u>0$,
$X(t+u)$ has the same law as the variable obtained by ranking in
decreasing order the sequences $X^{(1)}(u), X^{(2)}(u),\ldots $ where
the latter are independent, random mass partitions with values in
$\nabla_1$ having the same distribution as $X(u)$ but scaled in size
by the factors under $s_1, s_2, \ldots ,$ respectively.

The construction of such processes, known as homogeneous fragmentation
processes, requires some care and was essentially carried out by
Bertoin \cite{Be2,Be3} (see also \cite{berest1}). We defer a brief
overview of the general construction to the next section. It is enough
here to note that $\nu(ds)$ is the ``rate'' at which given fragments
split into pieces whose relative sizes are given by $s=(s_1,s_2,\ldots
).$ This measure must verify the integrability condition $ \int
_{\nabla_1} (1-s_1)\nu(ds) <\infty.$ Some information about $\nu$
is captured by the function
%
%e1.1 ###
\begin{eqnarray}\label{phi}
\Phi(q):= \int_{\nabla_1}  \Biggl( 1-\sum_{i=1}^{\infty} s_i^{q+1}
 \Biggr)\nu(ds),  \qquad  q>\underline{p},
\end{eqnarray}
where
\[
\underline{p} := \inf \Biggl\{ p \in\R\dvtx \int_{\nabla_1}
\sum_{i=2}^{\infty} s_i^{p+1} \nu(ds) < \infty \Biggr\} \le0.
\]
%
%Observe that $\underline{p} \le0.$
From
now on, we always assume that $\underline p < 0.$
As we shall reveal in more detail later, the function $p \mapsto\Phi
(p)$ turns out to be the Laplace exponent of a natural subordinator
associated to the fragmentation.
%, i.e. $t \mapsto- \log|X_U(t)|,$ where $X_U(t)$ is the so-called
%tagged fragment.
Hence, $\Phi$ is strictly increasing, concave and smooth such that
$\Phi(0)=0.$
The equation
\[
(p+1)\Phi'(p)=\Phi(p)
\]
on $p>\underline{p}$ is known to have a unique solution in $(0, \infty
)$ which we shall denote by $\overline{p}$; cf. \cite{Be4}. Moreover,
$(p+1)\Phi'(p)- \Phi(p)>0$ when $p\in(\underline p, \overline p)$.
This implies that the function
%
%e1.2 ###
\begin{equation}
c_p := \frac{\Phi(p)}{p+1}
\label{E: cp}
\end{equation}
reaches its unique maximum on
$(\underline{p},\infty)$ at $\overline{p}$ and this maximum is equal to
$\Phi'(\overline{p}).$

%jb changed the following paragraph

%AKv34 following para out.
%Homogeneous fragmentation processes will be formally defined as Markov
%processs taking their values in $\mathcal{P}=\{\text{partitions of } \N
%enumeration of the "sizes" of the fragments (i.e. the blocks of the
%partition) at time $t$ (more details are given in Section \ref{S:
%disc}). We define the filtration $\mathbb{F} = \{\mathcal{F}_t : t\geq
%0\}$ where $\F_t := \sigma\{ |\Pi(u)|, u \le t\}$ to be the natural
%filtration of $|\Pi(\cdot)|.$

%s1.2 ###
\subsection{The fragmentation FKPP equation and traveling
waves}\label{S: KPP analogue}

The main aim of this paper is to formulate the analogue of the
Fisher--Kolmogorov--Petrovskii--Piscounov (FKPP) equation for
fragmentation processes and thereby to analyze the existence,
uniqueness and asymptotics of its traveling waves.

Consider a homogeneous fragmentation process $\Pi$ with dislocation
measure $\nu$.
The equation
%
%e1.3 ###
\begin{equation}
-c\psi'(x) + \int_{\nabla_1}\biggl\{\prod_{i}\psi(x- \log s_i) -\psi
(x)\biggr\}\nu(ds) =0, \qquad    x\in\mathbb{R},\label{E: integ dif equ-first-view}
\end{equation}
is the \textit{fragmentation traveling wave} equation with wave speed
$c\in\mathbb{R}.$ Equation~(\ref{E: integ dif equ-first-view}) is
the analogue in the fragmentation context of the classical traveling
wave equation associated with the FKPP equation. This is discussed in
greater detail in Section \ref{S: disc} as well as existing results of
the same flavor concerning BBM and BRWs. However, let us momentarily
note the connection of (\ref{E: integ dif equ-first-view}) with a
currently open problem and hence, the motivation for the current work.

A classical result, due to Bramson \cite{bram}, establishes that the
growth, as a function of time $t$, of the right-most particle in a
one-dimensional, unit rate, binary splitting BBM is $\sqrt{2}t -
3\cdot2^{-3/2}\log t + O(1)$. The precise quantification of this
result is done through a particular traveling wave solution of the FKPP
equation. On account of the technical similarities between BBM and
fragmentation processes, there are many reasons to believe that a
similar result should also hold in the latter setting. Indeed, it is
already known that the largest block at time $t>0$, $X_1(t)$, satisfies
a strong law of large numbers in the sense that
\[
\lim_{t\uparrow\infty} \frac{-\log X_1(t)}{t} = c_{\overline{p}},
\]
almost surely. Therefore, a natural conjecture, motivated by Bramson's
result, is that there exists a deterministic function of time $(\gamma
(t)\dvtx t\geq0)$ such that
\[
\mathbb{P}\bigl(-\log X_1(t) - \gamma(t)\geq x\bigr)\rightarrow\psi
_{\overline{p}}(x),
\]
where $\psi_{\overline{p}}(x)$ is a solution to (\ref{E: integ dif
equ-first-view}) with $c=c_{\overline{p}}$ and moreover
\[
\gamma(t) = c_{\overline{p}}t -\theta\log t +O(1)
\]
for some constant $\theta$.

We do not address the above conjecture in this paper. Instead, we shall
study the existence, uniqueness and asymptotics of the solution of this
\textit{fragmentation traveling wave} equation. Below we introduce our
main result in this direction. We first introduce two classes of functions.
\begin{definition}
The class of functions $\psi\in C^1(\mathbb{R})$ such that $\psi
(-\infty)=0$ and $\psi(\infty) = 1$ and $\psi$ is monotone
increasing is denoted by $\mathcal{T}_1$. For each $p \in(\underline
{p},\op]$ we further define $\mathcal{T}_2(p) \subset\mathcal{T}_1$
as the set of $\psi_p \in\mathcal{T}_1$ such that $ L_p(x):=
e^{(p+1)x} (1 -\psi_p(x) )$ is monotone increasing.
\end{definition}

Let $X$ be a homogenous mass fragmentation process as in the previous
section and let $\mathbb{F} = \{\mathcal{F}_t \dvtx t\geq0\}$ where $\F
_t := \sigma\{ X(u), u \le t\}$ is its natural filtration. Our main
result follows.

\begin{theorem} \label{T: main thm}
Fix $p \in(\underline{p},\op]$ and suppose that $\psi_p\dvtx\mathbb
{R}\mapsto(0,1]$ belongs to $ \TT_2(p).$ Then
%
%e1.4 ###
\begin{equation}\label{E:prod mart}
M(t,p,x) := \prod_{i} \psi_p\bigl(x -\log X(t) -c_p t\bigr), \qquad
t\geq0,
\end{equation}
is a $\mathbb{F}$-martingale if and only if $\psi_p$ solves (\ref{E: integ dif equ-first-view}) with $c=c_p$.
Furthermore, up to an additive translation, there is only one such
function $\psi_p\in\TT_2(p)$ which is given by
%
%e1.5 ###
\begin{equation}\label{E: unique candidate for psi}
\psi_p(x) = \E\bigl(\exp\bigl\{-e^{-(p+1)x} \Delta_p\bigr\}\bigr),
\end{equation}
where $\Delta_p>0$ a.s. is an $\F_{\infty}$-measurable martingale
limit (see Theorem \ref{T:mgcgce} and Definition~\ref{deltadef}
for more details).
\end{theorem}

 Note that, $M(t,p,x)$ is a martingale, then it is
necessarily uniformly integrable since it is bounded in $[0,1]$ and,
therefore, converges in $L^1$ to its limit, which we denote by $M(\infty
, p, x)$.

%s2 ###
\section{Discussion}\label{S: disc}\label{sec2}

%s2.1 ###
\subsection{Homogeneous fragmentations}\label{sec2.1}

We start by stating the definition of and stating some results concerning
homogeneous fragmentations.

The construction and manipulation of general homogeneous fragmentations
are best carried out in the framework of partition valued
fragmentations. Let $\PP=\{ \mbox{partitions of } \N \}.$ An element
$\pi$ of $\PP$ can be identified with an infinite collection of
blocks (where a block is just a subset of $\N$ and can be the empty
set), $\pi= (B_1,B_2,\ldots )$ where $\bigcup_i B_i =\N$, $B_i \cap B_j
=\varnothing$ when $i \neq j$ and the labeling corresponds to the order of the
least element, that is, if $w_i$ is the least element of $B_i$ (with
the convention $\min\varnothing=\infty$), then $i \le j \Leftrightarrow w_i
\le w_j.$ The reason for such a choice is that we can discretize the
processes by looking at their restrictions to $[n]:=\{1,\ldots ,n\}$
(if $\pi\in\PP$, we denote by $\pi_{|[n]}$ the natural partition it
induces on $[n]$).
Roughly speaking, a homogeneous fragmentation is a $\PP$-valued
process $[\Pi(t),t\ge0]$ such that blocks split independently of each
other and with the same intensity. Given a subset $A=\{a_1,a_2,\ldots \}
$ of $\N$ and a partition $\pi=(B_1,B_2,\ldots ) \in\PP$, we
formally define the \textit{splitting of $B$ by $\pi$} to be the
partition of $B$ (i.e., a collection of disjoint subsets whose union is
$B$) defined by the equivalence relation $a_i \sim a_j$ if and only if
$i$ and $j$ are in the same block of $\pi.$

\begin{definition}  \label{fragdef}
Let $\Pi=(\Pi(t),t\ge0)$ be a $\PP$-valued Markov process with c\`adl\`ag\footnote{Continuity is
understood with respect to the following
metric. The distance\vspace*{-2pt} between two partitions, $\pi$ and $\pi'$, in
$\mathcal{P}$, is defined to be $2^{-n(\pi, \pi')}$ where
$n(\pi, \pi')$ is the largest\vspace*{-2pt} integer such that $\pi_{|[n]} = \pi
'_{|[n]}$.} sample paths. $\Pi$ is called a homogeneous fragmentation
if its semi-group has the following, so-called, \textit{fragmentation
property}: For every $t,t' \ge0$ the conditional distribution of $\Pi
(t+t')$ given $\Pi(t)=\pi$ is that of the collection of blocks one
obtains by splitting the blocks $\pi_i, i=1,\ldots, $ of $\pi$ by an
i.i.d. sequence $\Pi^{(i)}(t')$ of exchangeable random partitions
whose distribution only depends on $t'$. We also impose the condition that $\Pi(0)$
is the trivial partition with a single block made up of the whole set
$\N$ (see \cite{Be3} for further discussion).
\end{definition}

Given $\pi=(\pi_1,\pi_2,\ldots ) \in\PP$, we say that it has
asymptotic frequencies if for each $i$
\[
|\pi_i| =\lim_{n \to\infty} \frac{\# \pi_i \cap[n]}{n}
\]
exists. We write $|\pi| = (|\pi_i|, i \in\N)$ for the decreasing
rearrangement of the frequencies of the blocks. It is known that if
$\Pi$ is a homogeneous fragmentation, then almost surely, for all $t
\ge0, \Pi(t)$ has asymptotic frequencies. The process $[|\Pi(t)|,
t\ge0]$ is called a mass fragmentation and coincides with the process
described in the opening section. One can define directly mass
fragmentations (i.e., Markov processes with state space $\nabla_1$
such that each fragment splits independently with the same rate) but
all such processes can be seen as the frequency process of an
underlying integer partition fragmentation.

Given the ``stationary and independent increments'' flavor of the
definition of a fragmentation process, it is no surprise that there is
a L\'evy--Khintchin type description of the law of these processes.
Bertoin \cite{Be3} shows that the distribution of a homogeneous
fragmentation $\Pi$ is completely characterized by a sigma-finite
measure $\nu$ on
%jb changed the following
$\nabla^{(-)}_1 := \{s_1 \ge s_2 \ge\cdots \ge0 , \sum_i s_i \le1\}
$ (\textit{the disclocation measure}) which satisfies
%
%e2.1 ###
\begin{equation}\label{E:integ cond de nu}
\int_{\nabla^{(-)}_1} (1-s_1)\nu(ds) <\infty
\end{equation}
and a parameter $\mathtt{c} \ge0$ (\textit{the erosion rate}).

The meaning of the dislocation measure $\nu$ is best understood
through a Poissonian construction of the process $\Pi$ which is given
in \cite{Be2}.
%jb added this
%AEKv34 took it out!!! For simplicity we suppose that there is no
%erosion ($c=0$).
Roughly speaking, for each label $i\in\N$ we have an independent
Poisson point process on $\R_+ \times\nabla_1^{(-)}$ with intensity
$dt \otimes\nu(ds).$
If $(t_k, s_k)$ is an atom of the point process with label $i$, then at
time $t_k$ the block $B_i(t_k-)$ of $\Pi(t_k-)$ is split into
fragments of relative size given by $s_k$ (or more precisely, it is
split according to Kingman's paintbox partition directed by $s_k$). To
make the construction rigorous, one needs to show that the point
processes can be used to construct a compatible family of
Markov chains, $(\Pi^{(n)}(t), t\ge0)$, each of which lives on
$\PP_n$, the space of partitions of $\{1,\ldots, n\}$. Hence, $\nu
(ds)$ is the rate at which blocks independently fragment into
subfragments of relative size $s.$

The role of the erosion term ${\mathtt c}$ is easier to explain in
terms of mass fragmentations. Indeed, %jb added this
if $\Pi$ is a $(\nu,0)$-fragmentation, then
the process
\[
e^{-{\mathtt c}t} |\Pi(t)|=(e^{-{\mathtt c}t} |\Pi_1(t)|,
e^{-{\mathtt c}t} |\Pi_2(t)|, \ldots ), \qquad    t\geq0,
\]
is a $(\nu,{\mathtt c})$-mass fragmentation. The erosion is
essentially a deterministic phenomenon.
%, unless otherwise specified we
%will always suppose ${\mathtt c}=0.$
The dislocation measure $\nu$ thus plays the same role as in our
introductory example where $\nu$ was finite.

In the present work we will always suppose that the dislocation measure
$\nu$ is \textit{conservative}, that is, $\operatorname{supp}(\nu) \subseteq
\nabla_1$. Moreover, we assume there is \textit{no erosion}, namely
$\mathtt{c}=0$.

Recall the definition
\[
\Phi(q)= \int_{\nabla_1}  \Biggl( 1-\sum_{i=1}^{\infty} s_i^{q+1}
 \Biggr)\nu(ds), \qquad   q>\underline{p},
\]
and that the function $p \mapsto\Phi(p)$ is the Laplace exponent of a
pure jump subordinator associated to the fragmentation. This
subordinator is precisely the process $t \mapsto- {\log}|\Pi_1(t)|.$
It is also not difficult to show that the associated L\'evy measure is
given by
%
%e2.2 ###
\begin{equation}
\label{m(dx)}
m(dx) := e^{-x}\sum_{i=1}^\infty\nu(-\log s_i \in dx).
\end{equation}

%s2.2 ###
\subsection{Branching processes, product martingales and the FKPP equation}\label{sec2.2}

Our main theorem follows in the spirit of earlier results which concern
BBM or BRWs. We discuss these here and we explain why equation (\ref
{E: integ dif equ-first-view}) is the analogue of the FKPP traveling
wave equation for fragmentations.

The classical FKPP equation, in its simplest form, takes the form of
the nonlinear parabolic differential equation
%
%e2.3 ###
\begin{equation}
\frac{\partial}{\partial t} u(x,t)= \frac{1}{2}\,\frac{\partial^2
}{\partial x^2}u(x,t) + u(x,t)^2 - u(x,t).
\label{E: classical}
\end{equation}
It is one of the simplest nonlinear partial differential equation which
admits a traveling wave solution of the form $u(t,x)=\psi(x+ct)$ where
$\psi\dvtx\R\mapsto[0,1]$ and $c$ is the speed of the traveling wave. It
is a classical result that a traveling wave solution to (\ref{E:
classical}) exists with speed $c$ if and only if $|c|\geq\sqrt{2}$,
in which case it solves
%
%e2.4 ###
\begin{eqnarray}
\label{E: classicalTW}
\tfrac{1}{2}\psi''(x)- c \psi'(x) +\psi^2(x)-\psi(x) &=&0,\nonumber
\\[-8pt]
\\[-8pt]
\psi(-\infty)&=&0 ,  \qquad    \psi(\infty)=1.
\nonumber
\end{eqnarray}
%
%jb changed the sign in front of c

There are many different mathematical contexts in which the FKPP
equation appears, not least of all within the original setting of gene
concentrations; cf. \cite{KPP,F}. However, McKean \cite
{mckeanapplicationbbmkpp}, Ikeda, Nagasawa and Watanabe \cite{INW1,INW2,INW3},
Neveu \cite{neveu} and Chauvin \cite{chauvin} all show
that the FKPP equation has a natural relationship with one-dimensional
dyadic BBM. This process consists of an individual particle initially
positioned at zero that executes a Brownian motion until a unit-mean,
independent and exponentially distributed time at which point it
divides into two particles. Each of these particles behaves
independently of one another and their parent and undergoes the same
life-cycle as the initial ancestor but space--time shifted to its point
of creation. The process thus propagates via the obvious iteration of
this procedure.%\looseness=1

Let $(V_1(t), V_2(t), \ldots , V_{N(t)})$ denote the positions of the
particles at time $t$ in the BBM.
Neveu \cite{neveu} (see also \cite{chauvin}, Theorem~3.1) shows that
$\psi\dvtx\R\mapsto[0,1]$ is a solution of (\ref{E: classicalTW}) %jb
%this citation doesn't work, it refers to E: classical prod mart ???wtf
%???
with speed $|c|\geq\sqrt{2}$ if and only if
%
%e2.5 ###
\begin{equation}
M_t = \prod_{i=1}^{N(t)} \psi\bigl(x+V_i(t) -c t \bigr)
\label{E: classical prod mart}
\end{equation}
is a martingale for all $x\in\R$.

A natural extension of this model is to replace the BBM by a
continuous-time BRW. This process can be described as follows: A
particle initially positioned at the origin lives for a unit mean,
exponentially distributed length of time and upon dying it scatters a
random number of offsprings in space relative to its own position
according to the point process whose atoms are $\{\zeta_i, i=1,\ldots
\}$. Each of these particles then iterates this same procedure,
independently from each other and from the past, starting from their
new positions. As stated above, let $(V_1(t), V_2(t), \ldots ,
V_{N(t)})$ denote the positions of the particles at time $t$.

This was, for instance, studied by Kyprianou in \cite{kyp99} where it
was shown that
in this context the analogue of (\ref{E: classicalTW}) %jb idem above.
is
%
%e2.6 ###
\begin{eqnarray}
\label{E: BRW TW }
-c \psi'(x) +  \biggl( E  \biggl[\prod_{i} \psi(\zeta_i +x)  \biggr]
- \psi(x)  \biggr) &=&0, \nonumber
\\[-8pt]
\\[-8pt]
 \psi(-\infty)&=&0 , \qquad   \psi(\infty)=1
\nonumber
\end{eqnarray}
and that
%
%e2.7 ###
\begin{equation}
M_t = \prod_{i=1}^{N(t)} \psi\bigl(x+V_i(t) -ct\bigr)
\label{E: BRW prod mart}
\end{equation}
is a martingale for all $x\in\R$ if and only if $\psi$ is a solution
to (\ref{E: BRW TW }) with speed $c$.

In the case of discrete time BRW, (\ref{E: BRW TW }) %jb idem, prob ref
reduces to the so-called smoothing transform and a similar result holds
(see, e.g., \cite{biggins,BK97,kyp88,liu,durlig}).

Given that homogeneous fragmentations can be seen as generalized BRWs,
it is natural to formulate analogous results in this wider setting.
To derive the analogue of (\ref{E: classical}), let us consider the
following intuitive reasoning where, for simplicity, we assume that
$\nu$ is a finite dislocation measure.

The classical technique for solving (\ref{E: classical}) with initial
condition $u(x,0) = g(x)$ consists of showing that
\[
u(x,t): = \mathbb{E} \Biggl(\prod_{i=1}^{N(t)} g\bigl(x+V_i(t)\bigr) \Biggr)
\]
%
%jb was - is now + in def of u above
is a solution. We take the same approach for fragmentations, where the
role of the position $V_i(t)$ is played by $\{-{\log}|\Pi_i(t)|\dvtx i\geq
1\}$. It is easily seen from the fragmentation property in Definition
\ref{fragdef} that
\[
u(x,t+h) = \mathbb{E} \Biggl(\prod_{i=1}^{N(t)} u\bigl(x-{\log}|\Pi_i(h)|,
t\bigr) \Biggr).
\]
Recall also from the definition of a fragmentation process that in each
infinitesimal period of time $h$, each block independently experiences
a dislocation, fragmenting into smaller blocks of relative size $(s_1,
s_2, \ldots )\in\nabla_1$ with probability $\nu(ds)h +o(h)$ and more
than one split dislocation occurs with probability $o(h)$. Roughly
speaking, it follows that as $h\downarrow0$,
\begin{eqnarray*}
&& u( x, t+h) - u(x, t)\\
&& \qquad =
\mathbb{E} \biggl(\prod_i u\bigl(x-{\log}|\Pi_i(h)|, t\bigr) \biggr)- u( x, t)
\\
&& \qquad =\int_{\nabla_1}
 \biggl\{
\prod_{i}u(x - \log s_i,t )-u(x,t)
 \biggr\}\nu(ds) h +o(h).
\end{eqnarray*}
This suggestively leads us to the parabolic integro-differential equation
%
%e2.8 ###
\begin{equation}
\frac{\partial u}{\partial t}(x, t) = \int_{\nabla_1} \biggl \{ \prod
_i u(x - \log s_i , t) - u(x,t) \biggr\}\nu(ds)
\label{E: defn-FKPP}
\end{equation}
with initial condition $u(x,0) = g(x)$,
which we now formally identify as the analogue to the FKPP equation,
even for the case that $\nu$ is an infinite measure.

A traveling wave $\psi\dvtx \R\mapsto[0,1]$ of (\ref{E: defn-FKPP})
with wave speed $c\in\mathbb{R}$, that is, $u(x,t) = \psi(x+ct)$,
therefore solves
%
%e2.9 ###
\begin{equation}
-c\psi'(x) + \int_{\nabla_1}\biggl\{\prod_{i}\psi(x- \log s_i) -\psi
(x)\biggr\}\nu(ds) =0, \qquad    x\in\mathbb{R}.\label{E: integ dif equ}
\end{equation}
%
%jb changed the sign in front of c

%s2.3 ###
\subsection{Additive martingales}\label{sec2.3}

In the setting of both BBM and BRW, a key element to the analysis of
(\ref{E: classicalTW}) and (\ref{E: BRW TW }) has been the study of
certain ``additive'' martingales; cf. \cite
{mckeanapplicationbbmkpp,neveu,chauvin,BK97,Ha,kypAIHP}.
In the current setting, these martingales are defined as follows:
\begin{definition}
For all $p >\underline{p}$ we define the \textit{additive martingale}
\[
W(t ,p ):= e^{\Phi(p)t}\sum_{i}|\Pi_i(t)|^{p+1},  \qquad   t\geq0,
\]
and the (critical) \textit{derivative martingale}
\[
\partial W(t,{\op}) := -\sum_{i} \bigl(t\Phi'({\op}) +{\log}|\Pi
_i(t)| \bigr)e^{\Phi({\op})t}|\Pi_i(t)|^{{\op}+1},  \qquad   t\geq0.
\]
\end{definition}

The fact that the processes in the above definition are martingales is
not difficult to establish; we refer the reader to Bertoin and Rouault
\cite{BeRo} for further details. As we shall see,
consistently with the case of BBM and BRW, it is the limits of these
two martingales for certain parameter regimes in $p$ which play a
central role in the analysis of
(\ref{E: integ dif equ}).
Note that the additive martingale is positive and therefore, converges
almost surely. The derivative martingale, so called because it is
constructed from the derivative in $p$ of the additive martingale, is a
signed martingale and it is not   a priori  clear that it converges
almost surely. Moreover, when limits for either of the two martingales
exist, it is also not clear if they are nontrivial. The following
theorem, lifted from Bertoin and Rouault \cite{BeRo}, addresses
precisely these questions.

\begin{theorem} \label{T:mgcgce}
{\smallskipamount=0pt\begin{longlist}[(ii)]
\item[(i)] If $p\in(\underline{p}, {\op})$ then $W(t,p)$ converges
almost surely and in mean to its limit, $W(\infty, p)$, which is
almost surely strictly positive. If $p\geq{\op}$ then $W(t,p)$
converges almost surely to zero.
\item[(ii)] If $p={\op}$ then $\partial W(t,{\op})$ converges almost
surely to a nontrivial limit, $ \partial W(\infty, {\op})$, which is
almost surely strictly positive and has infinite mean.
\end{longlist}}
%We define $\Delta_p=W(\infty,p)$ if $p<\overline p $ and $\Delta_{
\end{theorem}

\begin{definition}\label{deltadef}
Henceforth, we shall define $\Delta_p = W(\infty,p)$ if $p\in
(\underline{p}, \overline{p})$ and $\Delta_{\overline p} = \partial
W(\infty, \overline{p})$.
\end{definition}

%s2.4 ###
\subsection{\texorpdfstring{Idea of the proof of Theorem \protect\ref{T: main
thm}}{Idea of the proof of Theorem 1}}
\label{sec2.4}

Let us briefly discuss in informal terms the proof of Theorem~\ref{T: main thm}. For
convenience, we shall define the integro-differential operator
\[
\mathcal{A}_p\psi(x) = -c_p \psi'(x) + \int_{\nabla_1}\biggl\{\prod
_{i}\psi(x- \log s_i) -\psi(x)\biggr\}\nu(ds)
\]
whenever the right-hand side is well defined.
%jb added this
Observe that the second term in the operator $\mathcal{A}_p$ is a jump
operator driven by $\nu$ and is thus very close to the generator of
the fragmentation itself.
%AKv34 following sentence added.
Indeed, $\mathcal{A}_p$ is the generator of the fragmentation process
with dislocation measure $\nu$ and erosion coefficient $c_p$ (albeit
that the latter may take negative values).

Moreover, we shall say that $\psi_p\in\TT_2(p)$ is a \textit{multiplicative martingale function} if
(\ref{E:prod mart}) is a martingale. The equivalence of the analytical
property $\mathcal{A}_p\psi_p \equiv0$ and the probabilistic
property of $\psi_p$ being a multiplicative martingale function
emerges from a classical Feynman--Kac representation as soon as one can
apply the appropriate stochastic calculus (which in the current setting
is necessarily driven by the underlying Poisson random measure used to
define the fragmentation process) in order to give the semi-martingale
representation of
$M(t,p,x)$.

To show the unique form of solutions in $\TT_2(p)$, we start by
studying the asymptotics of multiplicative martingale functions. We
shall elaborate slightly here for the case $p \in(\underline{p},\op)$.
To this end, consider $\psi_p \in\TT_2(p)$ a multiplicative
martingale function which makes (\ref{E:prod mart}) a martingale. As
$M$ is a uniformly integrable martingale $\psi_p(x) = M(0,p ,x)
=E(M(\infty,p,x)).$ Our objective is, therefore, to understand in more
detail the martingale limit $M(\infty, p, x)$. Taking logs of the
multiplicative martingale we see that
\[
-\log M(t, p, x) = - \sum_i \log\psi_p \bigl(x-{\log}|\Pi_i(t)| -c_p t\bigr).
\]
Recall that $ L_p(x)= e^{(p+1)x} (1 -\psi_p(x) ).$ Hence, if we
replace $-\log x$ by $(1-x)$ (assuming that the arguments of the log
are all asymptotically close to 1 for large~$t$) and multiply by
$e^{(p+1)x}$, we see that
%
%e2.10 ###
\begin{equation}
 \hspace*{30pt} -e^{(p+1)x} \log M(t, p, x)
\approx
\sum_i |\Pi_i(t)|^{p+1}e^{\Phi(p)t } L_p\bigl(x-{\log}|\Pi_i(t)| -c_p t\bigr)
\label{E: make the sub}
\end{equation}
as $t\uparrow\infty$.
We know, however, that the $\sum_i |\Pi_i(t)|^{p+1}e^{\Phi(p)t } $
is the additive martingale $W(t,p)$ which converges almost surely. On
the other hand, for each fixed $i \in\N$ we know that $-\log(|\Pi
_i(t)|)-c_p t$ is Bertoin's tagged fragment subordinator minus some
drift so it is a L\'evy process with no negative jumps. Accordingly,
under mild assumptions, it respects the law of large numbers
\[
-\log(|\Pi_i(t)|)-c_p t\sim \alpha_p t
\]
as $t\uparrow\infty$ for some constant $\alpha_p$ which can be shown
to be positive.
Heuristically speaking, it follows that, if we can substitute
$L(x+\alpha_p t)$ in place of $L_p(x-{\log}|\Pi_i(t)| -c_p t)$ in
(\ref{E: make the sub}), then we can ``factor out'' the $L_p$ terms to get
\[
\frac{-e^{(p+1)x} \log M(t, p, x)}{W(t,p)} \approx L_p(x+\alpha_p t)
\]
as $t\uparrow\infty$
and since both martingale limits $M(\infty, p, x)$ and $W(\infty, p)$
are nontrivial, we deduce that $\lim_{z\to\infty} L_p(z) =k_p$ for
some constant $k_p\in(0,\infty)$. A direct consequence of this is
that, irrespective of the multiplicative martingale function $\psi_p$
that we started with, the $L^1$ limit $M(\infty, p, x)$ is always
equal (up to an additive constant in $x$) to $\exp\{ - e^{-(p+1)x}
W(\infty, p)\}$. The stated uniqueness of $\psi_p$ follows immediately.

Clearly this argument cannot work in the case $p=\overline p $ on
account of the fact that Theorem \ref{T:mgcgce} tells us that
$W(\infty, p)=0$ almost surely. Nonetheless, a similar, but more
complex argument in which a comparison of the $L^1$ martingale limit
$M(\infty, \overline p ,x)$ is made with the derivative martingale
limit $\partial W(\infty, \overline p )$ is possible.

One particular technical difficulty in the above argument is that one
needs to uniformly control the terms $L_p (x-{\log}|\Pi_i(t)| -c_p t)$
in order to ``factor out'' the common approximation $L_p(x+\alpha_p t)$.
One could try to control the position of the left-most particle to do
this (as in \cite{Ha}), however, this turns out to be inconvenient and
another technique, namely, projecting the martingales along the
so-called stopping-lines, appears to work better. This is a classical
idea in the context of BBM (see, e.g., Neveu \cite{neveu} and Kyprianou
\cite{kypAIHP}) and BRW (see, e.g., Biggins and Kyprianou \cite{BK97,BK2005}).

Thanks to their uniform integrability, the martingales
$M(t,p,x)$ and\break $W(t, p,x)$ may be seen as the projection of their limits
on to $\mathcal{F}_t$, the information generated by the fragmentation
tree when it is ``cut'' at fixed time $t$. However, we could also
project these limits back on to filtrations generated by the
fragmentation tree up to different ``increasing'' sequences of ``cuts''
or \textit{stopping lines} as they turn out to be known.
To give an example, a particular instance of an ``increasing sequence''
of stopping line is studied by Bertoin and Martinez \cite{BeMa} who
\textit{freeze} fragments as soon as their size is smaller than a certain
threshold $e^{-z}, z\ge0$. After some time all fragments are smaller
than this threshold and the process stops, thereby ``cutting'' through
the fragmentation tree. The collection of fragments one gets in the
end, say $(\Pi_i(\ell^z)\dvtx i \in\N)$,
generates another filtration $\{\mathcal{G}_{\ell^z} \dvtx z\geq0\}$.
Projecting the martingale limits $M(\infty, p, x)$ and $W(\infty, p,
x)$ back on to this filtration produces two new uniformly integrable
martingales which have the same limits as before and which look the
same as $M(t, p, x)$ and $W(t, p)$, respectively, except that the role
of $(\Pi_i(t)\dvtx i \in\N)$ is now played by $(\Pi_i(\ell^z) \dvtx i \in
\N)$. Considering the case $p=0$, so that $c_p = 0$, one could now
rework the heuristic argument given in the previous paragraphs for this
sequence of stopping lines and take advantage of the uniform control
that one now has over the fragment sizes on $(\Pi_i(\ell^z)\dvtx i \in\N)$.
A modification of this line of reasoning with a different choice of
stopping line when $p\neq0$ can and will be used as a key feature in
the proof of our main theorem.

%s2.5 ###
\subsection{Organization of the paper}\label{sec2.5}

The rest of the paper is organized as follows. In Section \ref{S:
existence} we show that the function $\psi_p$ given by (\ref{E:
unique candidate for psi})
\[
\psi_p(x) = \E\bigl(\exp\bigl\{-e^{-(p+1)x} \Delta_p\bigr\}\bigr)
\]
indeed makes
\[
M(t,p,x) = \prod_{i} \psi_p\bigl(x -{\log}|\Pi_i(t)| -c_p t\bigr), \qquad
t\geq0,
\]
a martingale and that it belongs to $\TT_2(p)$.

Section \ref{S: stop line} introduces the notions and tools related to
martingale convergence on stopping lines. More precisely, we generalize
the notion of \textit{frozen fragmentation} defined in \cite{BeMa} to
other stopping lines which allow us to study the whole range of wave
speeds. One particular feature which will emerge in this section will
be the fact that, for certain parameter ranges, the stopping lines we
consider will sweep out a Crump--Mode--Jagers (CMJ) process embedded
within the fragmentation process.

A central result of the proof is given in Section \ref{S: LLN} where
we provide a law of large numbers for the empirical distribution of the
sizes of the blocks on the stopping lines defined in Section~\ref{sec4}. A key
tool here is the connection between these stopped fragmentations and
the aforementioned embedded CMJ processes which allows us to use
classical results from Nerman \cite{nerman}.

Next, in Section \ref{S: exact}, we show that $\psi_p(x) = \E(\exp
(-e^{-(p+1)x} \Delta_p)) $ %where $\Delta_p = W(\infty, p)$ if $p\in(
are the only functions in $\TT_2(p)$ for $(\underline p, \overline p]$
which make $M$ a product martingale.
We conclude by showing that on the one hand this function $\psi_p$
solves (\ref{E: integ dif equ}) and on the other hand that, if $\psi
_p \in\TT_2(p)$ solves (\ref{E: integ dif equ}), then it makes $M$ a
martingale.

%s3 ###
\section{Existence of multiplicative martingale functions}\label{S:
existence}

Our main objective here is to establish the range of speeds $c$ in
(\ref{E:prod mart}) for which multiplicative martingale functions exist.
In short, the existence of multiplicative martingale functions follows
directly from the existence of a nontrivial limit of the martingales
$W(\cdot, p)$ for $p\in(\underline p, \overline p)$ and $\partial
W(\cdot, \overline p)$. We begin with a classification of possible
speeds. We shall call wave speeds $c$ \textit{sub-critical} when $c\in
(c_{\underline{p}}, c_{{\op}})$, \textit{critical} when $c= c_{{\op}}$
and \textit{super-critical} when $c>c_{{\op}}$.
\eject

\begin{theorem} \label{existence}\label{T: existence}
{\smallskipamount=0pt
\begin{longlist}[(ii)]
\item[(i)] At least one multiplicative martingale function [which
makes $(\ref{E:prod mart})$ a martingale] exists in $\mathcal{T}_1$
(the set of montone, $C^1$ functions with limit $0$ in $-\infty$ and
$1$ in $+\infty)$ for all wave speeds $c\in(c_{\underline{p}},
c_{{\op}}]$.
\item[(ii)] For all wave speeds $c>c_{{\op}}$ there is no
multiplicative martingale function in $\mathcal{T}_1.$
\end{longlist}}
\end{theorem}

%jb in many cases $(c_{\underline{p}}=-\infty.$ In the case $(c_{

Before proceeding with the proof, we make the following observation.

\begin{rem}\label{rem1}
%jb added the following
As we suppose that $\nu$ is conservative, we have that $\Phi(0)=0$
and ${\op} >0, \underline p \in[-1,0].$ The map $p \mapsto\Phi(p)$
can have a vertical asymptote at $\underline p$ or a finite value with
finite or infinite derivative.

The above theorem does not consider the case $c < c_{\underline{p}}$.
Observe, however, that as soon as $ \lim_{p \searrow\underline{p} }
\Phi(p) = -\infty$, we have $c_{\underline{p}} = - \infty$, so this
case is void. When $\Phi(\underline{p} +) >-\infty$, we have to look
at the right-derivative of $\Phi$ at $\underline{p}$. If $\Phi
'(\underline{p}+)<\infty$, then the fragmentation is necessarily
finite activity (i.e., $\nu$ has finite mass) and dislocations are
always finite (see \cite{berest2}). Hence, we are dealing with a usual
continuous-time BRW for which the results are available in the
literature. We leave open the interesting case where $\Phi(\underline
{p} +) >-\infty$ but $\Phi'(\underline{p}+)=\infty.$

The constants ${\op}$ and $ \underline{p}$ also play an important
role in \cite{berest2}. More precisely, following Bertoin \cite{Be4}
it was shown in \cite{berest2} that one could find fragments decaying
like $t \mapsto e^{-\lambda t}$ provided that $\exists p \in
(\underline{p}, {\op}]$ (the case $p=\underline{p}$ being
particular) such that $\lambda= \Phi'(p).$ Hence, the set of
admissible wave speeds $\{ c_p \dvtx p \in(\underline p, {\op}] \}$ is
closely related to the set of admissible speeds of fragmentation.
However, the only $p$ such that $c_p = \Phi'(p)$ is ${\op}.$
\end{rem}

\begin{pf*}{Proof of Theorem \ref{existence}} First, let us assume
that $c<c_{{\op}}$.
Since
\[
W(t+s, p) = \sum_{i}e^{\Phi(p)t}|\Pi_i(t)|^{p+1}W^{(i)}(s,p ),
\]
where $W^{(i)}(\cdot, p )$ are i.i.d. copies of $W(\cdot )$ independent
of $\Pi(t)$, it follows that taking Laplace transforms and then limits
as $s\uparrow\infty$ and that if we take $\psi_p$ as in (\ref{E:
unique candidate for psi})
\[
\psi_p(x) = \mathbb{E}\bigl(\exp\bigl\{-e^{-(p+1)x} W(\infty, p)\bigr\}\bigr),
\]
then
%
%e3.1 ###
\begin{eqnarray}
\label{prob0}
\psi_p(x) &=& \mathbb{E}\biggl(\prod_{i}\exp\bigl\{ -e^{-(p+1)x} e^{\Phi
(p)t}|\Pi_i(t)|^{p+1}W^{(i)}(\infty,p )\bigr\}\biggr)\nonumber
\\[-8pt]
\\[-8pt]
&=&
\mathbb{E}\prod_{i}\psi_p\bigl(x- {\log}|\Pi_i(t)| -c_pt\bigr).
\nonumber
\end{eqnarray}
Hence, we see that $\psi_p$ is a multiplicative martingale function
with wave speed $c_p.$

Note the fact that $\psi_p\in C^1(\mathbb{R})$ follows by dominated
convergence.
Note also that from the definition of $\psi_p$ it follows
automatically that $\psi_p(\infty)=1$. Moreover, we know from Theorem
\ref{T:mgcgce} that $\mathbb{P}(W(\infty, p )>0)=1$ so $\psi
_p(-\infty)=0$ and $\psi_p \in\TT_1.$

Next, assume that $c=c_{{\op}}$. The method in the previous part of
the proof does not work because, according to the conclusion of Theorem
\ref{T:mgcgce}(i), $W(\infty,{\op})=0$ almost surely. In that case,
it is necessary to work instead with the derivative martingale. Using
the conclusion of Theorem \ref{T:mgcgce}(ii) we note that by
conditioning on the $\mathcal{F}_t$, with $\partial W^{(i)}(\infty,
{\op})$ as i.i.d. copies of $\partial W(\infty, {\op})$, we have
appealing to similar analysis to previously
\begin{eqnarray*}
\partial W(\infty, {\op})& =& -\sum_{i} \bigl(t\Phi'({\op}) +{\log}
|\Pi_i(t)| \bigr)e^{\Phi({\op})t}|\Pi_i(t)|^{{\op}+1}
W^{(i)}(\infty, {\op}) \\
&&{}+ \sum_{i}e^{\Phi({\op})t}|\Pi_i(t)|^{{\op}+1}\,\partial
W^{(i)}(\infty, {\op})\\
&=&\sum_{i}e^{\Phi({\op})t}|\Pi_i(t)|^{{\op}+1}\,\partial
W^{(i)}(\infty, {\op}),
\end{eqnarray*}
where the second equality follows, since $W(\infty, {\op})=0$ and, as
above, the Theorem~\ref{T:mgcgce} tells us that $\mathbb{P}(\partial
W(\infty, \overline p)>0)=1$.

Now assume that $c>c_{{\op}}$. The following argument is based on
ideas found in \cite{Ha}.
Suppose that $|\Pi_1(t)|^\downarrow$ is the largest fragment in the
process $|\Pi|$. Then, as alluded to in the \hyperref[intro]{Introduction}, we know that
%
%e3.2 ###
\begin{equation}
\lim_{t\uparrow\infty} \frac{-{\log}|\Pi_1(t)|^\downarrow}{t} =
\Phi'({\op} ) = c_{\op}
\label{left-most-speed}
\end{equation}
almost surely. See \cite{Be4,berest2}.

Suppose that a multiplicative martingale function $\psi\in\mathcal
{T}_1$ exists within this regime. Set $c >c_{{\op}}$. It follows by
virtue of the fact that $\psi$ is bounded in $(0,1]$ that
%
%e3.3 ###
\begin{equation}
\psi(x)=\mathbb{E}\prod_{i} \psi\bigl(x- {\log}|\Pi_i(t)| -ct\bigr)\leq
\mathbb{E}\psi\bigl(x-{\log}|\Pi_1(t)|^\downarrow -ct\bigr)
\label{takelimits}
\end{equation}
for all $t\geq0$. From Remark~\ref{rem1} on the rate of decay of $|\Pi
_1(t)|^{\downarrow}$, we can easily deduce that
\[
-{\log}|\Pi_1(t)|^\downarrow -ct \rightarrow-\infty
\]
almost surely as $t\uparrow\infty$ since $ c_{{\op}}-c<0$. Taking
limits in (\ref{takelimits}) as $t\uparrow\infty$, we deduce by
dominated convergence and the fact that $\psi(-\infty)=0,$ that $\psi
(x)=0$ for all $x\in\mathbb{R}$. This contradicts the assumption that
$\psi$ is in $\TT_1.$
\end{pf*}

Recall that $\TT_2 (p)= \{ \psi\in\TT_1 \dvtx L_p(x) =e^{(p+1)x}
(1-\psi(x))  $ is monotone increas\-ing$  \}.$

\begin{proposition}\label{P: in T2}
Fix $p \in(\underline{p},\op]$. Define $\psi_p(x) = \mathbb
{E}(\exp\{-e^{-(p+1)x}\Delta_p\})$,
% as in (\ref{E: unique candidate
%for psi}) recalling that $\Delta_p = W(\infty, p)$ if $p\in(
then  $\psi_p \in\TT_2(p).$
\end{proposition}

\begin{pf}
We have already established that $\psi_p \in\TT_1.$ Consider $\psi
_p(x)= \mathbb{E}\exp\{ - e^{-(p+1)x } \Delta_p\}$ for some
nonnegative and nontrivial random variable $\Delta_p$. Write $y =
e^{-(p+1)x}$. From Feller \cite{feller}, Chapter XIII.2, it is known
that $[1-\psi_p(-(1+p)^{-1}\log y)]/y$ is the Laplace transform of a
positive measure and hence, is decreasing in $y$. In turn, this implies
that $L_p(x)$ is an increasing function in $x$.
%\rightqed
\end{pf}

%s4 ###
\section{Stopping lines and probability tilting} \label{S: stop line}\label{sec4}

The concept of \textit{stopping line} was introduced by Bertoin \cite
{Be2} in the context of fragmentation processes, capturing in its
definition the essence of earlier ideas on stopping lines for branching
processes coming from the work of Neveu \cite{neveu}, Jagers \cite
{jagers}, Chauvin \cite{chauvin} and Biggins and Kyprianou \cite
{BK97}. Roughly speaking, a stopping line plays the role of a stopping
time for BRWs. The tools and techniques we now introduce also have an
intrinsic interest in themselves and cast a new light on some earlier
results by Bertoin.

%s4.1 ###
\subsection{Stopping lines}\label{sec4.1}

The following material is taken from \cite{Be2}.
Recall that for each integer $i \in\N$ and $s \in\R_+$, we denote
by $B_i(s)$ the block of $\Pi(s)$ which contains $i$ with the
convention that $B_i(\infty)=\{i\}$ while $\Pi_i(t)$ is the $i$th
block by order of least element. Then we write
\[
\G_i(t) = \sigma\bigl(B_i(s) , s\le t\bigr)
\]
for the sigma-field generated by the history of that block up to time
$t$. We will also use the notation $\G_t = \sigma\{ \Pi(s), s\le t\}
$ for the sigma-field generated by the whole (partition-valued)
fragmentation. Hence, obviously, for all $t\ge0$ we have $\F_t :=
\sigma\{ |\Pi(u)|, u \le t\} \subset\G_t$ and for each $i \in\N,
\G_i(t) \subset\G_t.$

\begin{definition}
We call stopping line a family $\ell=(\ell(i) , i\in\N)$ of random
variables with values in $[0, \infty]$ such that for each $i \in\N$:
\begin{longlist}[(ii)]
\item[(i)] $\ell(i)$ is a $(\G_i(t))$-stopping time.
\item[(ii)] $\ell(i)=\ell(j)$ for every $j \in B_i(\ell(i)).$
\end{longlist}
\end{definition}

For instance, first passage times such as $\ell(i) = \inf\{ t\ge0 \dvtx
|B_i(t)| \le a \}$ for a fixed level $a\in(0,1)$ define a stopping line.

The key point is that it can be checked that the collection of blocks
$\Pi(\ell)=\{ B_i(\ell(i)) \}_{i \in\N}$ is a partition of $\N$
which we denote by $\Pi(\ell) = (\Pi_1(\ell), \Pi_2(\ell),\ldots
),$ where, as usual, the enumeration is by order of least element.

Observe that because $B_i(\ell(i))=B_j(\ell(j))$ when $j \in B_i(\ell
(i))$, the set\break  $\{ B_i(\ell(i)) \}_{i \in\N}$ has repetitions, $(\Pi
_1(\ell), \Pi_2(\ell),\ldots )$ is simply a way of enumerating each
element only once by order of discovery. In the same way the $\ell
(j)$'s can be enumerated as $\ell_i, i=1,\ldots, $ such that for each
$i, \ell_i$ corresponds to the stopping time of $\Pi_i(\ell).$

If $\ell$ is a stopping line, it is not hard to see that both $\ell+t
:= (\ell(i)+t, i \in\N)$ and $\ell\wedge t := (\ell(i) \wedge t, i
\in\N)$ are also stopping lines. This allows us to define
\[
\Pi\circ\theta_\ell(t) : = \Pi(\ell+t)
\]
and the sigma-field $\G_\ell= \bigvee_{i \in\N} \G_i(\ell(i)).$

The following lemma (see \cite{Be2}, Lemma~3.13) can be seen as the
analogue of the strong Markov property for branching processes; it is
also known as the Extended Fragmentation Property; cf. Bertoin \cite
{Be3}, Lemma~3.14. %jb to check
\begin{lemma}\label{L:Markov}
Let $\ell=(\ell(i), i\in\N)$ be a stopping line, then the
conditional distribution of $\Pi\circ\theta_\ell\dvtx t \mapsto\Pi
(\ell+t)$ given $\G_\ell$ is $\P_{\pi}$ where $\pi= \Pi(\ell).$
\end{lemma}

Heuristically, we are going along the rays from the root to the
boundary, one at a time (each integer defines a ray, but observe that
there are some rays which do not correspond to an integer). On each ray
$\xi$ we have a stopping time $\tau_{\xi}$, that is, we look only at
what happens along that ray $(\xi_t, t\ge0)$ and based on that
information, an alarm rings at a random time. When later we go along
another ray $\xi'$, if $\xi'(\tau_\xi)=\xi(\tau_\xi)$ (i.e., the
two rays have not branched yet at $\tau_\xi$), then $\tau_{\xi
'}=\tau_\xi.$

Following Chauvin \cite{chauvin} and Kyprianou \cite{kypptrf} we now
introduce the notion of \textit{almost sure dissecting} and $L^1$-\textit{dissecting} stopping lines.

\begin{definition}
Let $\ell=(\ell(i) , i\in\N)$ be a stopping line.
\begin{longlist}[(ii)]
\item[(i)] We say that $\ell$ is a.s. dissecting if almost surely
$\sup_i \{\ell_i\} <\infty$.
\item[(ii)]Let $A_\ell(t) = \{ i \dvtx \ell_i > t \},$ then we say that
$\ell$ is $p-L^1$-dissecting if
\[
\lim_{t \to\infty} \mathbb{E} \biggl(\sum_{i \in A_\ell(t)} |\Pi
_i(t)|^{p+1} e^{\Phi(p) t} \biggr) =0.
\]
\end{longlist}
\end{definition}

%s4.2 ###
\subsection{Spine and probability tilting}\label{sec4.2}

In this section we will discuss changes of measures and subsequent path
decompositions which were instigated by Lyons \cite{lyons} for the BRW
and further applied by Bertoin and Rouault \cite{BeRo} in the setting
of fragmentation processes. The following lemma is a so-called
many-to-one principle. It allows us to transform expectations of
functionals along a stopping line into expectations of functions of a
single particle, namely Bertoin's tagged fragment.
\begin{lemma}[(Many-to-one principle)]\label{L:many to one}
Let $\ell$ be a stopping line. Then, for any measurable nonnegative
$f$ we have
%
%e4.1 ###
\begin{equation}\label{E:many-to-one}
\E\biggl(\sum_i |\Pi_i(\ell)| f(|\Pi_i(\ell)|, \ell_i)\biggr) = \E\bigl(f(|\Pi
_1(\ell)|,\ell_1) \mathbf{1}_{\{ \ell_1 <\infty\}}\bigr).
\end{equation}
\end{lemma}

\begin{pf}
To see this, observe that because $\Pi(\ell)$ is an exchangeable
partition, the pair $(\Pi_1(\ell),\ell_1)$ is a size-biased pick
from the sequence $((\Pi_i(\ell),\ell_i), i\in\N)$ [i.e., given
$((|\Pi_i(\ell)|,\ell_i), i\in\N)= ((x_i, l_i), i\in\N)$ then
$\P( (|\Pi_1(\ell)|,\ell_1) =(x_i,l_i))=x_i$]. The indicator
function in the right-hand side comes from the possibility that there
is some dust in $\Pi(\ell)$.
\end{pf}

Observe that if $\ell$ is almost surely dissecting, as $\Pi$ is
conservative, then $\Pi(\ell)$ has no dust. The converse is not true.

The second tool we shall use is a probability tilting that was
introduced by Lyons, Permantle and Peres \cite{LPP} for Galton--Watson
processes, Lyons \cite{lyons} for BRWs and by Bertoin and Rouault
\cite{BeRo} for fragmentation.
First note that since $(-{\log}|\Pi_1(t)|\dvtx t\geq0)$ is a subordinator
with Laplace exponent $\Phi$, it follows that%\looseness=1
\[
\mathcal{E}(t,p) :=|\Pi_1(t)|^p e^{t\Phi(p)}
\]
is a $(\P,\G_1(t))$-martingale. If we project this martingale on the
filtration $\mathbb{F}$, we obtain the martingale $W(t,p).$ We can use
these martingales to define the tilted probability measures
%
%e4.2 ###
\begin{equation}
d\P^{(p)}_{|\G_t} = \mathcal{E}(t,p)\,d\P_{|\G_t}
\label{E: change of pb}
\end{equation}
and
\[
d\P^{(p)}_{|\F_t} = W(t,p)\,d\P_{|\F_t}.
\]

The effect of the latter change of measure is described in detail in
\cite{BeRo}, Proposition~5. More precisely, under $\P^{(p)}$ the
process $\xi_t = -{\log}|\Pi_1(t)|$ is a subordinator with Laplace exponent
\[
\Phi^{(p)}(q) =\Phi(p+q) - \Phi(p), \qquad    q > \underline{p} -p,
\]
and L\'evy measure given by
%
%e4.3 ###
\begin{equation}
m^{(p)}(dx) = e^{-px}m(dx),
\label{mp}
\end{equation}
where $m(dx)$ was given in (\ref{m(dx)}).
Under $\P^{(p)}$, the blocks with index not equal to unity split with
the same dynamic as shown before. The block containing 1 splits
according to the atoms of Poisson Point Process $\{(t,\pi(t)) \dvtx t\geq
0\}$ on $\R_+ \times\mathcal{P} $ with intensity $dr \otimes|\pi
_1|^p \nu(d\pi).$

%jb changed below to precise $\ell_1$ is $\G_1(t) measurable
Observe that because $\ell_1=\ell(1)$ is a $\G_1(t)$-stopping time,
it is measurable with respect to the filtration of the aforementioned
Poisson Point Process and thus the above description is enough to
determine the law of $\ell_1$ under $\P^{(p)}.$ Adopting the notation
$\tau=\ell_1$ and writing $\xi_t$ instead of $-{\log}\vert\Pi_1(t)
\vert$, we have the following result:
\begin{lemma}\label{L: proba tilt} For all positive measurable $g$,
%
%e4.4 ###
\begin{equation}\label{E:tiltingspine} \qquad
\E\biggl( \sum_i g(-{\log}|\Pi_i(\ell)|,\ell_i) |\Pi_i(\ell)|^{p+1}
e^{\Phi(p)\ell_i}\biggr) =\E^{(p)}\bigl( g(\xi_{\tau}, \tau) \mathbf{1}_{\{
\tau<\infty\}}\bigr).
\end{equation}
\end{lemma}
\begin{pf} From Lemma \ref{L:many to one} with $f(x,\ell) = x^p
g(x,\ell)e^{\Phi(p)\ell}$ we have
\begin{eqnarray*}
&& \E\biggl( \sum_i g(-{\log}|\Pi_i(\ell)|,\ell_i) |\Pi_i(\ell)|^{p+1}
e^{\Phi(p)\ell_i}\biggr)\\
&& \qquad = \E\bigl(g(-{\log}|\Pi_1(\ell)|,\ell_1) |\Pi_1(\ell)|^p e^{\Phi
(p)\ell_1} \mathbf{1}_{\{\ell_1<\infty\}}\bigr)\\
&& \qquad =\E\bigl(g(\xi_\tau, \tau) \mathcal{E} (\tau, p) \mathbf{1}_{\{ \tau
<\infty\}} \bigr)
%jb replaced a t by \tau in \cE
\\ && \qquad =\E^{(p)}\bigl( g(\xi_{\tau}, \tau) \mathbf{1}_{\{ \tau<\infty\}}\bigr)
\end{eqnarray*}
and the result follows.
\end{pf}

As a first application of these tools we prove the analogue of Theorem~2 in
\cite{kypptrf} which gives a necessary and sufficient condition
for a stopping line to be $p-L^1$ dissecting.

\begin{theorem}\label{T:L1 dissect}
A stopping-line $\ell$ is $p-L^1$-dissecting if and only if
\[
\P^{(p)}(\ell_1 <\infty)=1.
\]
\end{theorem}

\begin{pf} From Lemma~\ref{L: proba tilt},
\begin{eqnarray*}
\E\biggl(\sum_{i \in A_\ell(t)} |\Pi_i(t)|^{p+1} e^{\Phi(p) t} \biggr) &=& \E
\biggl(\sum_{i} \mathbf{1}_{\{ \ell_i >t\}} |\Pi_i(t)|^{p+1} e^{\Phi(p)
t} \biggr) \\
&=& \Bbb{P}^{(p)}( \ell_1 >t ).
\end{eqnarray*}
It follows that the limit of the left-hand side as $t\uparrow\infty$
is zero if and only if $P^{(p)}(\ell_1 <\infty)=1$.
\end{pf}

%s4.3 ###
\subsection{Martingales}\label{sec4.3}

We define an ordering on stopping lines as follows: given $\ell^{(1)}$
and $\ell^{(2)}$ two stopping lines, we write $\ell^{(1)} \le\ell
^{(2)}$ if almost surely, for all $i \in\N\dvtx \ell^{(1)}(i) \le\ell
^{(2)}(i)$. So given a family $(\ell^z, z\ge0)$ of stopping lines we
say that $\ell^ {z}$ is increasing
if almost surely, for all $z \le z', \ell^z \le\ell^{z'}.$

Given $(\ell^z, z\ge0)$ an increasing family of stopping lines, we
may define two filtrations $\mathcal{G}_{\ell^z}$ and $\mathcal
{F}_{\ell^z}$ as follows:
\[
\mathcal{G}_i(\ell^z) = \sigma\bigl(B_i(s)\dvtx s\leq\ell^z(i) \bigr)  \quad
\mbox{and}\quad
\mathcal{F}_i(\ell^z) = \sigma\bigl(|B_i(s)|\dvtx s\leq\ell^z(i)\bigr)
\]
and then define
\[
\mathcal{G}_{\ell^z}: = \bigvee_{i}\mathcal{G}_i(\ell^z) \quad  \mbox{and} \quad \mathcal{F}_{\ell^z}: = \bigvee_{i}\mathcal{F}_i(\ell^z).
\]

Finally, we say that an increasing family of stopping lines $\ell^z$
is proper if $\lim_{z\to \infty} \mathcal{G}_{\ell^z} =\sigma\{ \Pi(t),
t\ge0\}$ which is equivalent to $\ell_i^z \to\infty$ for all $i$
almost surely.

The next lemma mirrors analogous results that were obtained for BBM by
Chauvin \cite{chauvin}. Recall that $c_p=\Phi(p)/(p+1)$.

\begin{theorem}\label{stopmgs} Fix $p\in(\underline{p}, {\op}]$ and
let $\psi_p$ be a solution to $(\ref{E:prod mart})$, that is, a
function that makes $(M(t,p,x),t\ge0)$ a martingale and is bounded
between $0$ and $1$. Then:
\begin{longlist}[(ii)]
\item[(i)] Let $(\ell^z, z\ge0)$ be a increasing family of a.s.
dissecting lines. Then the stochastic process\vspace*{-1pt}
\[
M(\ell^z, p, x): = \prod_{i} \psi_p\bigl(x - {\log}|\Pi_i(\ell^z)| -c_p
\ell^z(i)\bigr), \qquad    z\geq0,
\]
is a uniformly integrable martingale with respect to $\{\mathcal
{F}_{\ell^z}\dvtx z\geq0\}$ having limit equal to $M(\infty, p,x)$.
\item[(ii)] Let $(\ell^z, z\ge0)$ be a increasing family of
$p-L^1$-dissecting lines, then the stochastic process\vspace*{-1pt}
\[
W(\ell^z, p) := \sum_{i}|\Pi_i(\ell^z)|^{p+1}e^{\Phi(p)\ell
^z(i)}, \qquad    z\geq0,
\]
is a unit mean martingale with respect to $\{\mathcal{F}_{\ell
^z}\dvtx z\geq0\}$. Furthermore, when $p\in(\underline{p}, {\overline p})$\vspace*{-1pt}
\[
\lim_{z\uparrow\infty} W(\ell^z, p) = W(\infty, p)
\]
in $L^1$
where $W(\infty, p)$ is the martingale limit described in Theorem
\ref{T:mgcgce}(\textup{i}).
\end{longlist}
\end{theorem}

\begin{rem}
A straightforward analogue result cannot hold for $\partial W$ as it is
a signed martingale on the stopping lines we consider.
\end{rem}

\begin{pf}
We start with (i). With the help of Lemma \ref{L:Markov}, we have for
all $x\in\mathbb{R}$
\begin{eqnarray*}
 &&\mathbb{E}\prod_{i} \psi_p\bigl(x - {\log}|\Pi_i(t)| - c_p t
\bigr) \\[-1pt]
&& \qquad = \mathbb{E}\prod_{i\dvtx\ell^z_i\geq t} \psi_p\bigl(x - {\log}|\Pi_i(t)|
-c_p t \bigr)
\\[-1pt]
 && \qquad  \quad {} \times \prod_{i\dvtx\ell_i^z <t} \prod_{j \dvtx \Pi_j(t) \subset\Pi
_i(\ell^z)} \psi_p \bigl(x -{\log}|\Pi_i(\ell^z)|- c_p \ell^z_i\\[-1pt]
&& \qquad  \quad
\hphantom{\times\prod_{i\dvtx\ell_i^z <t} \prod_{j \dvtx \Pi_j(t) \subset\Pi
_i(\ell^z)} \psi_p \bigl(\,}{} - \log
\bigl(|\Pi_j(t)|/|\Pi_i(\ell^z)|\bigr) -c_p (t- \ell^z_i)\bigr)\\
&& \qquad = \mathbb{E}\prod_{i\dvtx\ell^z_i\geq t} \psi_p\bigl(x - {\log}|\Pi_i(t)|
-c_p t \bigr)\\
&& \qquad  \quad {}\times \prod_{i\dvtx\ell_i^z <t} M^{(i)} \bigl(t- \ell^z_i, p,x-{\log}|\Pi
_i(\ell^z)|- c_p \ell^z_i\bigr)\\
&& \qquad = \mathbb{E}\prod_{i\dvtx\ell^z_i\geq t} \psi_p\bigl(x - {\log}|\Pi_i(t)|
-c_p t \bigr) \prod_{i\dvtx \ell_i^z <t} \psi_p\bigl(x - {\log}|\Pi_i(\ell^z) | -
c_p \ell^z_i\bigr),
\end{eqnarray*}
where in the second equalities, given $\mathcal{F}_{\ell^z}$, the
quantities $M^{(i)}(\cdot, p, \cdot)$ are independent copies of
$M(\cdot, p,\cdot)$. As $\ell^z$ is almost surely dissecting, we
know that as $t \to\infty$ the set $\{ i\dvtx\ell^z_i\geq t \}$ becomes
empty almost surely. Since $\psi_p$ is positive and bounded by unity,
we may apply dominated convergence to deduce as $t\uparrow\infty$
that for all $x\in\mathbb{R}$,
%
%e4.5 ###
\begin{equation}
\psi_p(x)=\mathbb{E}\prod_{i} \psi_p\bigl(x - {\log}|\Pi_i(\ell^z) | -
c_p \ell^z_i\bigr).
\label{constexp}
\end{equation}
Since $z\geq0$ is arbitrarily valued, the last equality in combination
with the Extended Fragmentation Property is sufficient to deduce the
required martingale property. Indeed, suppose that $z'>z\geq0;$ we
have that
\begin{eqnarray*}
&&M(\ell^{z'},p,x) \\
&& \qquad = \prod_{i} \psi_p\bigl(x - {\log}|\Pi_i(\ell^{z'})| -c_p \ell^{z'}_i\bigr)
\\ && \qquad= \prod_{i} \prod_{j \dvtx \Pi_j(\ell^{z'}) \subset\Pi_i(\ell
^z)} \psi_p\bigl(x - {\log}|\Pi_j(\ell^{z'})| -c_p \ell^{z'}_j\bigr)
\\ && \qquad= \prod_{i} \prod_{j \dvtx \Pi_j(\ell^{z'}) \subset\Pi_i(\ell
^z)} \psi_p\bigl( \bigl(x - {\log}|\Pi_i(\ell^z)| - c_p \ell^z_i\bigr)\\
&& \qquad  \quad\hphantom{ \prod_{i} \prod_{j \dvtx \Pi_j(\ell^{z'}) \subset\Pi_i(\ell
^z)} \psi_p\bigl( }
{} - \log\bigl(|\Pi
_j(\ell^{z'})|/|\Pi_i(\ell^z)|\bigr) -c_p (\ell^{z'}_j -\ell_i^z )\bigr).
\end{eqnarray*}
Using (\ref{constexp}) and Lemma \ref{L:Markov} we see that for all
$\mathcal{F}_{\ell^z}$-measurable $x'$,
\[
\psi_p (x') = \E\biggl( \prod_{j \dvtx \Pi_j(\ell^{z'}) \subset\Pi_i(\ell
^z)} \psi_p\bigl( x' - \log\bigl(|\Pi_j(\ell^{z'})|/|\Pi_i(\ell^z)|\bigr) -c_p
(\ell^{z'}_j -\ell_i^z )\bigr)\Big |\mathcal{F}_{\ell^z}\biggr)
\]
so the martingale property follows. Uniform integrability follows on
account of the fact that $0\leq\psi_p\leq1$ and since $\ell^z$ is
a.s. dissecting Lemma \ref{L:Markov} (and more precisely the
independence of the subtrees which start at $\ell^z$) gives us that
%jb added the dissecting property and tried to explain the use of Lemma
%1.
%
\[
\mathbb{E}(M(\infty, p,x)| \mathcal{F}_{\ell^z}) = M(\ell^z, p,x).
\]

We now prove (ii). Let $\ell$ be a $p-L^1$ dissecting stopping line.
By Lemma \ref{L: proba tilt} and the Monotone Convergence Theorem we
have that
%jb use of Lemma 3 new
%
\begin{eqnarray*}
\E(W(\ell,p)) &=& \lim_{t \to\infty} \E \biggl[ \sum_i \mathbf
{1}_{\{\ell_i \le t\}} |\Pi_i(\ell_i)|^{p+1} e^{\Phi(p) \ell
_i} \biggr] \\
&=&\lim_{t \to\infty} \E^{(p)} [ \mathbf{1}_{\tau\le t} ]\\
&=& 1,
\end{eqnarray*}
where we have used the many-to-one principle, the probability tilting
(\ref{E: change of pb}) and Theorem \ref{T:L1 dissect}.

To prove the martingale property, fix $0 \le z \le z'$ and observe that
\begin{eqnarray*}
W(\ell_z,p) &=& \sum_i \sum_{j \dvtx \Pi_j(\ell^{z'}) \subset\Pi
_i(\ell^z)} |\Pi_j(\ell^{z'})|^{p+1} e^{\Phi(p) \ell^{z'}_j}\\
&=& \sum_i |\Pi_i(\ell^{z})|^{p+1} e^{\Phi(p) \ell^{z}_i} \sum_{j
\dvtx \Pi_j(\ell^{z'}) \subset\Pi_i(\ell^z)}  \biggl(\frac{|\Pi
_j(\ell^{z'})|}{|\Pi_i(\ell^{z})|} \biggr)^{p+1} e^{\Phi(p) (\ell
^{z'}_j-\ell^z_i)}.
\end{eqnarray*}
We apply the strong Markov property in Lemma \ref{L:Markov} to obtain that
\[
\E \biggl( \sum_{j \dvtx \Pi_j(\ell^{z'}) \subset\Pi_i(\ell^z)}
 \biggl(\frac{|\Pi_j(\ell^{z'})|}{|\Pi_i(\ell^{z})|} \biggr)^{p+1}
e^{\Phi(p) (\ell^{z'}_j-\ell^z_i)}\Big | \mathcal{F}_{\ell
^{z}} \biggr) =1.
\]

The $L^1$ convergence of $W(\ell^z, p)$ when $p\in(\underline p,
\overline p)$ is a consequence of the fact that Lemma \ref{L:Markov}
applied to $\ell^z$ again gives us
\[
\mathbb{E}(W(\infty, p)| \mathcal{F}_{\ell^z}) = W(\ell^z, p)
\]
together with Theorem \ref{T:mgcgce}(i).
\end{pf}

%s4.4 ###
\subsection{First-passage stopping lines}\label{sec4.4}

In this paragraph, we introduce the families of stopping lines we will
be using and we show that they satisfy the desired properties.

Fix $p \in(\underline{p} , {\op}]$ and for each $z \ge0$ let $\ell
^{(p,z)}$ be the stopping line defined as follows: For each $i \in\N$
\[
\ell^{(p,z)}(i) = \inf\{ t \ge0 \dvtx -{\ln}|B_i(t)| > z + c_p t\},
\]
where $c_p=\Phi(p)/(p+1).$ In other words, $\ell^{(p,z)}(i)$ is the
first time when $-{\ln}|B_i(t)|$ crosses the line $x=c_p t + z$ (see
Figure \ref{stop line}). %jb prob ref. write "2" ?
Recall that $p \mapsto c_p$ is increasing on $(\underline p, \op]$
with $c_0=0.$

%f2 ###
\begin{figure}

\includegraphics{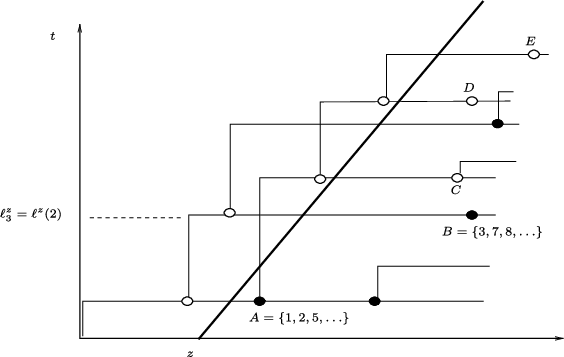}

\caption{Each dot represents a fragment (block) at its birth. The time
of birth is its vertical position while its $-\log$ size is given by
its position on the $x$ axis. Each horizontal line thus corresponds to
a single splitting event in which multiple fragments are created. The
blocks of $\Pi(\ell^{(p,z)})$ are the dots which are the first in
their line of descent to be on the right-hand side of the bold line $
x= z + c_p t .$ The collection of this dots is \textup{the coming
generation.} As $z$ increases, it reaches fragment $A$ which then
splits giving birth to $C,D,E, \ldots $ which will then replace $A$ in
the coming generation. For simplification this picture corresponds to
the case of a finite activity dislocation kernel $\nu$.}
%jb z+c_p t
\label{stop line}
\end{figure}

\begin{proposition}
For any fixed $p \in(\underline p, \op]$, the family of stopping
lines $(\ell^{(p,z)}, z\ge0)$ is a.s. dissecting and $p-L^1$
dissecting as well as proper.
\end{proposition}

\begin{pf}
%jb changed the proof of a.s. dissecting to include the $\bar p$ case.
We know from Theorem \ref{T:mgcgce} that when $p = \bar p$ the martingale
\[
W(t ,\bar p ):= e^{\Phi(\bar p)t}\sum_{i}|\Pi_i(t)|^{\bar p+1}, \qquad
t\geq0,
\]
converges almost surely to 0.
Hence, $e^{\Phi(\bar p) t} (|\Pi_1(t)|^\downarrow)^{\bar p+1} = \exp
( -( \bar p+1)\times\break  ( - \log(|\Pi_1(t)|^\downarrow) - c_{\bar p} t) )$
also tends to 0 a.s. and we conclude that
\[
\P\bigl( -\log( |\Pi_1(t)|^\downarrow) - c_{\bar p} t \to+\infty\bigr) =1.
\]
As $c_p < c_{\bar p}$ when $p <\bar p$, this entails immediately that
\[
\forall p \in(\underline{p}, \bar p] \qquad    \P\bigl( -\log( |\Pi
_1(t)|^\downarrow) - c_{ p} t \to+\infty\bigr) =1
\]
and hence, that
the stopping lines $(\ell^{(p,z)}\dvtx z\geq0)$ are a.s. dissecting.

To prove that for any $z \ge0, \ell^{(z,p)}$ is $L^1$-dissecting we
use Theorem \ref{T:L1 dissect}. Recall that under $\P^{(p)}$ the
process $(\xi_t = -{\log}|\Pi_1(t)|, t\ge0)$ is a subordinator with
Laplace exponent $\Phi^{(p)}.$
\[
\P^{(p)}\bigl(\ell^{(p,z)}_1<\infty\bigr) = \P^{(p)} (\inf\{t \dvtx\xi_t -c_p
t >z\}<\infty)
\]
and this is equal to one if and only if the mean of the L\'evy process
$(\xi_t - c_p , t\geq0)$ is positive under $\mathbb{P}^{(p)}$. That
is to say, if and only if $\Phi^{(p)\prime}(0)-c_p = \Phi'(p)
-c_p\ge0$ which is equivalent to $p \le\overline p .$
\end{pf}

Consider the stopping line $\ell^{(p,0)}.$ Two fundamental differences
in the way this stopping line dissects the fragmentation process occurs
in the regimes $p\le0$ and $p>0$. When $p \le0,$ it is easily seen
that for all $i \in\N,$ we have that $\ell^{(p,0)}(i)=0$ so that
$\Pi(\ell^{(p,0)})$ is the trivial partition with all integers in one
block. On the other hand, when $p>0,$ we claim that almost surely, for
all $i \in\N$ we have $\ell^{(p,0)}(i) >0.$ By exchangeability it is
enough to prove it for $\ell^{(p,0)}(1).$ Observing that $-{\log}|\Pi
_1(t)| - c_p t $ is a spectrally positive L\'evy process with negative
drift, standard theory (cf. \cite{Be1}, Chapter VII) tells us it has
the property that $(0,\infty)$ is irregular for $0$ which, in turn,
implies the claim. Hence, when $p>0$, the partition $\Pi(\ell
^{(p,0)})$ is a nontrivial collection of nonsingleton blocks.

%s4.5 ###
\subsection{Embedded CMJ process}\label{sec4.5}

%In this section we make an interesting observation concerning the
%process of blocks that are swept out of the underlying fragmentation
%process by the stopping lines $\{\ell^{(p,z)}: z\geq0\}$ and how they
%are fundamentally related to an underlying Crump-Mode-Jagers (CMJ)
%process. Our observation builds on an idea which goes back to Neveu

A CMJ process is a branching process in which a typical individual
reproduces at ages according to a random point process on $[0,\infty)$
and may or may not live forever.
The coming generation at time $t$ of a CMJ process consists of the
collection of individuals born after time $t$ whose parent was born
before time $t$.

In this section we show that, for appropriate values of $p$, the
collection of blocks $\Pi(\ell^{(p,z)})$ is also the coming
generation at time $z$ of certain CMJ process which is path-wise
embedded into the fragmentation process. Our observation builds on
ideas which go back to Neveu \cite{neveu} and Biggins and Kyprianou
\cite{BK97}.

%AKv34 vat ze xell is this para about? I deleted it.
%The case $p=0$ corresponds to a vertical line in Figure \ref{stop
%line} %jb prob ref, write "2" ?
%and has been studied in detail in Bertoin and Martinez \cite{BeMa}
%and Harris \textit{et al.} \cite{HKK}. Indeed, the process $z \mapsto\Pi(
%than some $\epsilon>0$ for the first time. It is evident that the
%final state of such a process corresponds exactly to the coming
%generation associated to the stopping line $\ell^{(0,z)}$ (with $p=0$)
%and $z= e^{-\epsilon}.$

In the following, we let $p \in(0,{\op}]$ be fixed and we consider
the collection of distances
\[
d_i = -\ln\bigl|\Pi_i\bigl(\ell^{(p,0)}\bigr)\bigr|- c_p \ell^{(p,0)}_i,
\]
that is, the point process of distances to the line $\ell^{(p,0)}$ of
the individuals in the first generation. Note specifically that the
latter point process cannot be defined when $p\leq0$ on account of the
fact that for all $i$ we have that $\ell^{(p,0)}(i)=0$.
Define $D^{(p)}(\cdot):= \sum_i \delta_{d_i} (\cdot)$ to be the
point process of the $d_i$'s and let $\mu^{(p)} =E[D^{(p)}]$ be its
intensity measure. The following proposition shows that the associated
intensity measure has several convenient properties.

\begin{proposition}\label{P:CMJ-conditions}
Fix $p \in(0,{\op}]$.
Then $\mu^{(p)}$ is a nonlattice measure with the following properties:
\begin{longlist}[(ii)]
\item[(i)] Its Malthusian parameter is equal to $p+1$, that is,
\[
\int_0^\infty e^{-(p+1) t } \mu^{(p)}(dt) =1.
\]
\item[(ii)] For all $\epsilon>0$ such that $|\Phi(p-\epsilon
)|<\infty$,
\[
\int_0^\infty e^{-(p+1-\epsilon) t} \mu^{(p)}(dt) <\infty.
\]
\end{longlist}
\end{proposition}

\begin{pf}
We first introduce some more notation: we define the \textit{martingale
weights}
%
%e4.6 ###
\begin{equation}
y_i^{(p)}\bigl(\ell^{(p,z)}\bigr): = \bigl|\Pi_i\bigl(\ell^{(p,z)}\bigr)\bigr|^{(p+1)}e^{\Phi
(p)\ell^{(p,z)}_i}.
\label{mgweights}
\end{equation}

Note that for any Borel set $A\in[0,\infty)$ we have, with the help
of the many-to-one principle in Lemma \ref{L: proba tilt}, that
\begin{eqnarray*}
%jb changed \alpha to p+1
\int_A e^{-(p+1) t} \mu^{(p)}(dt) &=& E\biggl( \sum_i e^{-(p+1) d_i}
\mathbf{1}_{\{d_i\in A\}}\biggr) \\
&=& E\biggl( \sum_i y_i^{(p)} \bigl(\ell^{(p,0)}\bigr) \mathbf{1}_{\{
-{\ln}|\Pi_i(\ell^{(p,0)})|- c_p \ell^{(p,0)}_i \in A
\}} \biggr)\\
&=& \mathbb{P}^{(p)}(Y_{\tau_0}\in A),
\end{eqnarray*}
where $Y_t= \xi_t -c_p t$ and $\tau_0= \inf\{ t \dvtx Y_t >0 \}$. It is
well known that, since $Y$ is spectrally positive, the law of $Y_{\tau
_0}$ is diffuse and hence, nonlattice.
Note also that
\[
\int_0^\infty e^{-(p+1) t} \mu^{(p)}(dt) = \E\Bigl(\sum y_i\bigl(\ell
^{(p,0)}\bigr)\Bigr)=\E(W(p,0))=1,
\]
which establishes the proof of part (i).

For the proof of part (ii), our objective is to compute
%
%e4.7 ###
\begin{equation}
\int_0^\infty e^{-(p+1-\epsilon)t} \mu^{(p)}(dt) = \mathbb
{E}^{(p)}(e^{\epsilon Y_{\tau_0}}).
\label{needtocheck}
\end{equation}
Noting that $(0,\infty)$ is irregular for $0$ for $Y$ (cf. \cite
{Be1}, Chapter~VII), and moreover, that $\mathbb{E}^{(p)}(Y_1) = \Phi
^{(p)\prime}(0)-c_ p = \Phi'(p)-c_p\geq0$, it follows that the
ascending ladder height of $Y$ is a compound Poisson process. The jump
measure of the latter, say $m^{(p)}_H(dx)$, is, therefore, proportional
to $\mathbb{P}^{(p)}(Y_{\tau_0}\in dx)$ and from \cite{kypbook},
Corollary~7.9, it can also be written in the form
\[
m_H^{(p)}(dx) = m^{(p)}(x,\infty)\,dx - \eta \biggl(\int_x^\infty
e^{-\eta(y-x)} m^{(p)}(y,\infty)\,dy \biggr)\,dx,
\]
where
$\eta$ is the largest root in $[0,\infty)$ of the equation $c_p\theta
- \Phi(\theta+p)+\Phi(p)=0$.

In order to verify (\ref{needtocheck}) it therefore suffices to prove that
\[
\int_0^\infty e^{\epsilon x}m^{(p)}_H(dx)<\infty
\]
whenever $|\Phi(p-\epsilon)|<\infty$. To this end, note that with
the help of Fubini's theorem and the fact that $m^{(p)}(dx) =
e^{-px}m(dx)$, we have
\begin{eqnarray*}
 &&\int_0^\infty e^{\epsilon x}m^{(p)}_H(dx)\\
&& \qquad = \int_0^\infty e^{\epsilon x}\int_x^\infty e^{-zp} m(dz)\,dx \\
&& \qquad  \quad {}- \eta
\int_0^\infty e^{\epsilon x} \biggl(\int_x^\infty e^{-\eta(y-x)} \int
_y^\infty e^{-pz}m(dz)\,dy \biggr)\,dx\\
&& \qquad =\int_0^\infty(1- e^{-\epsilon x}) e^{-(p-\epsilon)z} m(dz)\\
&& \qquad  \quad {}
-\frac{\eta}{\eta+\epsilon}\int_0^\infty(e^{\epsilon y} -
e^{-\eta y})  \biggl(\int_y^\infty e^{-pz} m(dz) \biggr)\,dy\\
&& \qquad = \frac{1}{\epsilon}[\Phi(p) - \Phi(p-\epsilon) ]-\frac{\eta
}{\epsilon(\eta+\epsilon)}\int_0^\infty
(1-e^{-\epsilon z})e^{-(p-\epsilon)z}m(dz) \\
&& \qquad  \quad {} + \frac{1}{\eta+\epsilon}\int_0^\infty(1- e^{-\eta
z}) e^{-pz} m(dz)\\
&& \qquad =\frac{1}{\eta+\epsilon}[\Phi(p) - \Phi(p-\epsilon)] + \frac
{1}{\eta+\epsilon}[ - \Phi(p)]\\
&& \qquad =\frac{1}{\eta+\epsilon}[ \Phi(p+\eta) - \Phi(p-\epsilon)],
\end{eqnarray*}
which is indeed finite when $|\Phi(p-\epsilon)|<\infty$.
\end{pf}

We conclude this section by showing an important relationship between
first passage stopping lines and the coming generation of an embedded
CMJ process. To this end, define for each $z \ge0$
\[
\mathcal{L}(z) := \bigl\{ \bigl|\Pi_i\bigl(\ell^{(p,z)}\bigr) \bigr|-c_p \ell
^{(p,z)}_i, i \in\N\bigr\}.
\]

\begin{theorem}\label{embedded-CMJ}
For each $p \in(0,\op],$ the process $z \mapsto\mathcal{L}(z)$ is
the process of the birth times of the individuals in the coming
generation at time $z$ of a CMJ process whose individuals live and
reproduce according to the point process $D^{(p)} .$
\end{theorem}

\begin{pf}
Note that Proposition \ref{P:CMJ-conditions}(i) implies that $\mu
^{(p)}(0,\epsilon)<\infty$ for all $\epsilon>0$ and hence, there is
an almost sure first atom in the point process $D^{(p)}$.
The theorem is trivially true for each $z \le\inf\{ d_i, i\in\N\}.$
The process $\Pi(\ell^{(p,z)})$ is constant on $[0,\inf\{d_i\})$ and
has a jump at time $z_1=\inf\{d_i\}$ where a single block splits.
Thanks to the fragmentation property, it gives birth to a collection of
blocks in $\Pi(\ell^{(p,z_1)})$ whose positions to the right of their
parent is again an instance of the point process $D^{(p)} .$ This shows
that as the line $\ell^{(p,z)}$ sweeps to the right, the coming
generation process $\mathcal{L}(z)$ describes a CMJ process which is
our claim.
\end{pf}

%s5 ###
\section{Laws of large numbers}\label{S: LLN}\label{sec5}

Before proceeding to the proof of asymptotics and uniqueness of
multiplicative-martingale functions in the class $\mathcal{T}_2(p)$,
we need to establish some further technical results which will play an
important role.
The following result is of a similar flavor to the types of laws of
large numbers found in Bertoin and Martinez \cite{BeMa} and Harris,
Knobloch and Kyprianou \cite{HKK}.

The next theorem gives us a strong law of large numbers for fragments
in $\Pi(\ell^{(p,z)})$ as $z\uparrow\infty$ with respect to the
weights (\ref{mgweights}) when $p\in(\underline p, 0]$. Recall that
the L\'evy measure $m$ was defined in (\ref{m(dx)}) and the definition
of the martingale weights
(\ref{mgweights}) gives
\[
y_i^{(p)}\bigl(\ell^{(p,z)}\bigr) = \bigl|\Pi_i\bigl(\ell^{(p,z)}\bigr)\bigr|^{(p+1)}e^{\Phi
(p)\ell^{(p,z)}_i}.
\]

\begin{theorem}\label{T: SandWLLN} Fix $p\in(\underline{p}, 0]$. Let
$f\dvtx[0,\infty)\rightarrow[0,\infty)$ such that $f(0)=0$.
Suppose $f(x)\leq Ce^{\epsilon x}$ for some $C>0$ and $\epsilon>0$
satisfying $|\Phi(p-\epsilon)|<\infty.$ Then
%
%e5.1 ###
\begin{eqnarray}
&&\lim_{z\uparrow\infty} \sum_i y_i^{(p)}\bigl(\ell^{(p,z)}\bigr) f \bigl(-\log
\bigl|\Pi_i\bigl(\ell^{(p,z)}\bigr)\bigr| -c_{p}\ell^{(p,z)}_i - z\bigr)\nonumber
\\[-8pt]
\\[-8pt]
&& \qquad  = Q^{(p)}(f)
{W(\infty, p)}
\label{intwosenses}
\nonumber
\end{eqnarray}
in probability where
%
%e5.2 ###
\begin{equation}
Q^{(p)}(f) = \frac{1}{\Phi'(p) -c_p} \int_{(0,\infty)}\biggl (\int
_0^y f(t)\,dt  \biggr)e^{-py}m(dy).
\label{smallpQp}
\end{equation}

If $f$ is uniformly bounded, then the above convergence may be upgraded
to almost sure convergence.
\end{theorem}

\begin{pf}
For convenience we shall define
\[
W\bigl(\ell^{(p,z)}, p, f\bigr)
:=
\sum_{i}y^{(p)}_i\bigl(\ell^{(p,z)}\bigr) f\bigl(-\log\bigl|\Pi_i\bigl(\ell^{(p,z)}\bigr)\bigr| -
c_p\ell^{(p,z)}_i -z\bigr).
\]
Note that from the many-to-one principle in Lemma \ref{L: proba tilt} that
%
%e5.3 ###
\begin{equation}
\mathbb{E}\bigl(W\bigl(\ell^{(p,z)}, p, f\bigr)\bigr) = \mathbb{E}^{(p)}\bigl(f(Y_{\tau_z} -z)\bigr),
\label{fromMT1}
\end{equation}
where $\tau_z = \inf\{t>0\dvtx Y_t >z\}$ and for $t\geq0$, $Y_t = \xi_t
- c_pt$. The process $Y$ is in fact a subordinator on account of the
fact that when $p\in(\underline p, 0]$, $c_p\leq0$. Moreover, it has
finite mean with $\mathbb{E}^{(p)}(Y_t) = (\Phi'(p) -c_p)t$ for
$t\geq0$. Note also that the assumption that $f(0)=0$ implies that the
expectation on the right-hand side of (\ref{fromMT1}) does not include
the possible contribution that comes from the event that $Y$ creeps
over $z$.

Let us first prove that if $f(x)\leq Ce^{\epsilon x}$ for some $C>0$
and $\epsilon>0$ satisfying $|\Phi(p-\epsilon)|<\infty$ (and, in
particular, for uniformly bounded $f$), we have that
%
%e5.4 ###
\begin{equation}
\lim_{z\uparrow\infty} \mathbb{E}\bigl(W\bigl(\ell^{(p,z)}, p, f\bigr)\bigr) = Q^{(p)}(f).
\label{simon's assumption}
\end{equation}
A classical result from the theory of subordinators (cf. \cite{Be1},
Chapter~3) tells us that for $y>0$
\[
\mathbb{P}^{(p)}(Y_{\tau_z} - z \in dy) = \int_{[0,z)}U^{(p)}(dx)
m^{(p)}(z-x+dy),
\]
where $U^{(p)}$ is the potential measure associated with $Y$ under
$\mathbb{P}^{(p)}$, meaning that $U^{(p)}(dx) = \int_0^\infty\mathbb
{P}^{(p)}(Y_t \in dx)\,dt$ and we recall that $m^{(p)}(dx) = e^{-px}
m(dx)$ is the L\'evy measure of $\xi$ under $\mathbb{P}^{(p)}$.
Hence, it follows that
\[
\mathbb{E}\bigl(W\bigl(\ell^{(p,z)}, p, f\bigr)\bigr) = U^{(p)}*g(z),
\]
where
\[
g(u) = \int_{(u,\infty)}f(y-u) m^{(p)}(dy).
\]

It can also be shown that $V(dx): = U^{(p)}(dx) + \delta_0(dx)$ is a
classical renewal measure of a renewal process with mean inter-arrival
time given by $\mathbb{E}^{(p)}(Y_1)$ (see \cite{kypbook}, Lemma~5.2).
The latter result also indicates that the associated
inter-arrival time of $V$ has distribution $\int_0^\infty e^{-t}
\mathbb{P}^{(p)}(Y_t \in dx)\,dt$ and hence, an easy computation shows
that the mean inter-arrival time is equal to $\mathbb
{E}^{(p)}(Y_1)=\Phi'(p)-c_p$. Applying the Key Renewal Theorem to
$V*g(z)$, we deduce that, whenever $g$ is directly Riemann integrable,
\begin{eqnarray*}
\lim_{z\uparrow\infty} \mathbb{E}\bigl(W\bigl(\ell^{(p,z)}, p, f\bigr)\bigr) &=& \frac
{1}{ \mathbb{E}^{(p)} (Y_1)}
\int_0^\infty g(t)\,dt\\ &= &\frac{1}{ \mathbb{E}^{(p)} (Y_1)} \int
_{(0,\infty)} \biggl(\int_0^y f(t)\,dt  \biggr)m^{(p)}(dy).
\end{eqnarray*}
Note that $g$ has no discontinuities on account of the fact that, for
each $u>0$, $g(u+)-g(u-) = f(0)m^{(p)}(\{u\})$ which equals zero thanks
to the assumption that $f(0)=0$.
Moreover, thanks to the assumption that $f(x)\leq C^{\epsilon x}$ and
$|\Phi(p-\epsilon)|<\infty$, we have that
\begin{eqnarray*}
\int_0^\infty g(t)\,dt & \leq&\frac{C}{\epsilon} \int_{(0,\infty)}
(e^{\epsilon y} - 1)m^{(p)}(dy) \\
&=& \frac{C}{\epsilon} \int_{(0,\infty)} (1 - e^{-\epsilon
y})e^{-(p-\epsilon)y}m(dy) \\
&=& \frac{C}{\epsilon}  \bigl( \Phi(p) - \Phi(p-\epsilon) \bigr)\\
&<&\infty,
\end{eqnarray*}
which shows that $g$ is directly Riemann integrable. We have thus
established (\ref{simon's assumption}).

Next we turn to establishing the limit (\ref{intwosenses}) in the
almost sure sense when $f$ is uniformly bounded. Harris, Knobloch and
Kyprianou \cite{HKK} show that, when $p=0$, the required strong law of
large numbers holds for all bounded measurable $f$ in the sense of
almost sure convergence. Although we are interested in conservative
fragmentation processes in this paper, the proof of (\ref
{intwosenses}) for the case that $p\in(\underline p, 0)$ is
mathematically similar to the dissipative case that was handled when
$p=0$ in \cite{HKK}. In the notation of \cite{HKK}, the role of the
quantity $X^{1+p^*}_{j,\eta}$ is now played by the martingale weights
$y_j(\ell^{(p,z)}).$ In that case, using (\ref{fromMT1}) in place of
the limit (9) in \cite{HKK}, all of the proofs go through verbatim, or
with obvious minor modification, with the exception of their Lemma~5
which incurs a moment condition. In fact, this moment condition turns
out to be unnecessary as we shall now demonstrate. The aforementioned
lemma requires us to show that, in the notation of the current setting,
%
%e5.5 ###
\begin{equation}
\sup_{z\geq0}\mathbb{E}(W(\ell^z, p)^q)<\infty \qquad \mbox{for some }q>1.
\label{no-moment}
\end{equation}
Thanks to Jensen's inequality, the process $(W(\ell^z, p)^q, z\geq0)$
is a submartingale. Hence, recalling that the almost sure limit of
$W(\ell^z, p)$ is $W(\infty, p)$, as soon as it can be shown that
$\mathbb{E}(W(\infty, p)^q)<\infty$ for some $q>1$, then (\ref
{no-moment}) is satisfied. Note however, the latter has been clearly
established in the proof of Theorem~2 of \cite{Be4} despite the fact
that the aforementioned theorem itself does not state this fact. This
completes the proof of the almost sure convergence (\ref{intwosenses})
for uniformly bounded~$f$.

We shall now obtain the required weak law of large numbers for $W(\ell
^{(p,z)}, p, f)$.
To this end, let us suppose that $\{f_k\dvtx k\geq1\}$ is an increasing
sequence of bounded positive functions such that, in the pointwise
sense, $f_k\uparrow f$. It follows from the aforementioned strong law
of large numbers for each $f_k$ that
\[
\liminf_{z\uparrow\infty} W\bigl(\ell^{(p,z)}, p, f\bigr) \geq\liminf
_{z\uparrow\infty}W\bigl(\ell^{(p,z)}, p, f_k\bigr) = Q^{(p)}(f_k)W(\infty,p)
\]
for all $k$ and hence, by monotone convergence,
\[
\liminf_{z\uparrow\infty} W\bigl(\ell^{(p,z)}, p, f\bigr) \geq
Q^{(p)}(f)W(\infty,p)
\]
almost surely.
Next note by Fatou's Lemma,
\begin{eqnarray*}
0&\leq& \mathbb{E}\Bigl(\liminf_{z\uparrow\infty} W\bigl(\ell^{(p,z)}, p,
f\bigr) - Q^{(p)}(f)W(\infty,p)\Bigr) \\
&=& \mathbb{E}\Bigl(\liminf_{z\uparrow\infty} W\bigl(\ell^{(p,z)}, p, f\bigr)\Bigr) -
Q^{(p)}(f) \\
&\leq& \liminf_{z\uparrow\infty} \mathbb{E}\bigl( W\bigl(\ell^{(p,z)}, p,
f\bigr)\bigr) - Q^{(p)}(f), \\
&=& 0,
\end{eqnarray*}
where the final equality follows by (\ref{simon's assumption}) and
hence, we are led to the conclusion that
%
%e5.6 ###
\begin{equation}
\liminf_{z\uparrow\infty} W\bigl(\ell^{(p,z)}, p, f\bigr) =
Q^{(p)}(f)W(\infty,p)
\label{a.s.liminf}
\end{equation}
almost surely.

Next define for $z\geq0$
\[
\Theta_z = W\bigl(\ell^{(p,z)}, p, f\bigr) - \inf_{u\geq z}W\bigl(\ell^{(p,u)}, p,
f\bigr) \geq0.
\]
Note by (\ref{simon's assumption}), monotone convergence and (\ref
{a.s.liminf}) that
\begin{eqnarray*}
\lim_{z\uparrow\infty}\mathbb{E}(\Theta_z) &=&\lim_{z\uparrow
\infty} \mathbb{E}\bigl(W\bigl(\ell^{(p,z)}, p, f\bigr)\bigr) - \lim_{z\uparrow\infty
}\mathbb{E}\Bigl(\inf_{u\geq z}W\bigl(\ell^{(p,u)}, p, f\bigr) \Bigr)\\
&=& Q^{(p)}(f) - \mathbb{E}\Bigl(\liminf_{z\uparrow\infty}W\bigl(\ell
^{(p,z)}, p, f\bigr) \Bigr)\\
&=&0.
\end{eqnarray*}
It now follows as a simple consequence from the Markov inequality that
$\Theta_z$ converges in probability to zero. Using the latter,
together with the almost sure convergence in (\ref{a.s.liminf}), it
follows that
\begin{eqnarray*}
&&W\bigl(\ell^{(p,z)}, p, f\bigr) - Q^{(p)}(f)W(\infty,p)\\
&& \qquad  = \Theta_z + \inf
_{u\geq z}W\bigl(\ell^{(p,u)}, p, f\bigr)- Q^{(p)}(f)W(\infty,p)
\end{eqnarray*}
converges in probability, to zero as $z\uparrow\infty$.
%jb why do you need slutsky here : you got \theta\to0 in proba and
%the second term goes to 0 a.s. so clearly the sum goes to 0 in proba
%no ?
\end{pf}

\begin{rem}
In the above proof, when dealing with the almost sure convergence (\ref
{intwosenses}) for uniformly bounded $f$, the replacement argument we
offer for Lemma~5 of Harris, Knobloch and Kyprianou \cite{HKK} applies
equally well to the case that the fragmentation process is dissipative.
This requires however, some bookkeeping based around the computations
in Bertoin \cite{Be4}, which assumes conservativeness, in order to
verify that the $L^q$ estimates are still valid. The consequence of
this observation is that the moment condition (A3) in \cite{HKK} is
unnecessary. In fact, the aforementioned condition (A3) is equivalent
to the condition that the dislocation measure has finite total mass in
that context.
\end{rem}

The next result gives us a strong law with respect to the weights (\ref
{mgweights}) for the regime $p\in(0, \overline p]$. For this we recall
also that the measure $D^{(p)} $ was defined for $p\in(0,\overline p]$
in Proposition \ref{P:CMJ-conditions} by
\[
D^{(p)}(dx) = \sum_{i}\delta_{\{-(p+1)^{-1}\log y_i(\ell^{(p,0)})
\in dx\}}
\]
and its intensity was denoted by $\mu^{(p)}$.
It is also worth recalling that the L\'evy measure associated with the
tagged fragment $\xi_\cdot: = -{\log}|\Pi_1(\cdot)|$ is denoted by
$m(dx)$ for $x>0$. Moreover, under the measure $\mathbb{P}^{(p)}$
where $p>\underline{p}$, the aforementioned L\'evy measure takes the
form $e^{-px}m(dx)$ for $x>0$.

\begin{theorem}\label{T: SLLN}
Fix $p\in(0,\op]$. Let $f\dvtx[0,\infty)\rightarrow[0,\infty)$ such
that $f(x)\leq Ce^{\epsilon x}$ for some $C>0$ and $\epsilon>0$
satisfying $|\Phi(p-\epsilon)|<\infty$. Then, almost surely,
%
%e5.7 ###
\begin{equation}
\lim_{z\uparrow\infty} \frac{\sum_i y_i^{(p)}(\ell^{(p,z)}) f
(-{\log}|\Pi_i(\ell^{(p,z)})| -c_{p}\ell^{(p,z)}_i - z) }{W(\ell
^{(p,z)}, p)} = Q^{(p)}(f),
\label{E: SLLN}
\end{equation}
where
%
%e5.8 ###
\begin{equation}
Q^{(p)}(f) = \frac{\int_0^\infty\int_u^\infty e^{-(p+1)y } f(y-u)
\mu^{(p)}(dy)\,du}{\int_{(0,\infty)} y e^{-(p+1)y } \mu^{(p)}(dy) }.
\label{bigpQp}
\end{equation}
\end{theorem}

\begin{pf} First recall from Theorem \ref{embedded-CMJ} that when
$p\in(0,{\op}]$, the sequence of stopping lines $(\ell^{(p,z)}\dvtx
z\geq0)$ sweeps out the coming generation of an embedded CMJ process
with Malthusian parameter $p+1$ and whose associated birth process is
described by the point process $D^{(p)} (\cdot)$.
For the aforementioned CMJ process we denote by $\{\sigma_i \dvtx i\geq0\}
$ the birth times of individuals, where the enumeration is in order of
birth times starting with the initial ancestor, counted as $0$, having
birth time $\sigma_0=0$. Define
\[
\phi^f_0(u) = \mathbf{1}_{\{u>0\}}e^{(p+1)u }\int_{(u,\infty)}
e^{-(p+1)y}f(y-u)D^{(p)}(dy).
\]
Then, following Jagers' classical theory of counting with
characteristics  (cf. Jagers \cite{jagers}), our CMJ processes have
count at time $z\geq0$ given by
\[
\eta^f_z := \sum_{i\dvtx \sigma_i \leq z}\phi^f_i(z- \sigma_i),
\]
where, for each $i$,
$\phi^f_i$ has the same definition as $\phi^f_0$ except that the
counting measure $D^{(p)} $ is replaced by the counting measure of the
point process describing the age of the $i$th individual at the moment
it reproduces. In this respect, the characteristics $\{\phi_i\dvtx i\}$
are i.i.d. Writing $\{\sigma_j^i \dvtx j\geq1\}$ for the ages at which
the $i$th individual reproduces and $\mathcal{C}_z$ for the index set
of individuals which form the coming generation at time $z\geq0$, we have
\begin{eqnarray*}
e^{-(p+1)z}\eta_z^f &=& \sum_{i\dvtx \sigma_i \leq z} e^{-(p+1)\sigma
_i}\sum_{j\dvtx \sigma^i_j > z-\sigma_i}
%jb the z was a t
e^{-(p+1)\sigma_j^i}f(\sigma^i_j + \sigma_i - z)\\
&=&\sum_{k\dvtx k\in\mathcal{C}_z} e^{-(1+p)\sigma_k}f(\sigma_k - z)\\
&=&\sum_i y_i^{(p)}\bigl(\ell^{(p,z)}\bigr) f\bigl (-\log\bigl|\Pi_i\bigl(\ell^{(p,z)}\bigr)\bigr|
-c_{p}\ell^{(p,z)}_i - z\bigr),
\end{eqnarray*}
%
%jb was c_\op in the last line
where the final equality follows from Theorem \ref{embedded-CMJ}. In
the particular case that $f$ is identically equal to unity, we denote
$\eta^f_z$ by $\eta^1_z$ and we see that $e^{-(p+1)z}\eta^1_z =
W(\ell^{(p,z)}, p)$.

Recall that $\mu^{(p)}$ was defined in Proposition \ref
{P:CMJ-conditions} as the intensity of the counting measure $D^{(p)} $.
The strong law of large numbers (\ref{E: SLLN}) now follows from the
classical strong law of large numbers for CMJ processes given in \cite
{nerman}, Theorem~6.3, which says that
%
%e5.9 ###
\begin{equation}
\lim_{z\uparrow\infty} \frac{\eta^f_z }{\eta^1_z} =
\frac{\int_0^\infty\int_u^\infty e^{-(p+1)y } f(y-u) \mu^{(p)}(dy)\,du}{\int_{(0,\infty)} y e^{-(p+1)y } \mu^{(p)}(dy) }
\label{nerman-limit}
\end{equation}
almost surely
provided that the following two conditions hold for some $\beta<p+1$. First,
\[
\int_{(0,\infty)} e^{-\beta z } \mu^{(p)}(dz) <\infty
\]
and second,
\[
\mathbb{E} \biggl(\sup_{u\geq0} e^{(p+1-\beta) u}\int_{(u,\infty)
} e^{-(p+1) y} f(y-u)D^{(p)}(dy) \biggr)<\infty.
\]
The first condition holds thanks to Proposition \ref
{P:CMJ-conditions}(ii) for all $\beta$ sufficiently close to $p+1$.
For the second condition, we note that when $f(x)\leq Ce^{\epsilon x}$,
we may estimate for all $\beta$ sufficiently close to $p+1$,
\begin{eqnarray*}
&&\mathbb{E} \biggl(\sup_{u\geq0} e^{(p+1-\beta) u}\int
_{(u,\infty) } e^{-(p+1) y} f(y-u)D^{(p)}(dy)  \biggr) \\
&& \qquad \leq C \mathbb{E} \biggl(\sup_{u\geq0} e^{(p+1-\beta-\epsilon)
u}\int_{(u,\infty) } e^{-(p+1-\epsilon) y} D^{(p)}(dy) \biggr)\\
&& \qquad \leq C\mathbb{E} \biggl(\int_{(0,\infty) } e^{-(p+1-\epsilon) y}
D^{(p)}(dy) \biggr)\\
&& \qquad  = C\int_{(0,\infty) } e^{-(p+1-\epsilon) y} \mu^{(p)}(dy),
\end{eqnarray*}
which is finite, again thanks to Proposition \ref{P:CMJ-conditions},
providing $|\Phi(p-\epsilon)|<\infty$.
Note in particular that these conditions also ensure that the
right-hand side of (\ref{nerman-limit}) is positive but neither zero
nor infinity in value.
The proof of part (i) of the theorem is thus complete as soon as we
note that (\ref{nerman-limit}) is the desired limit.
\end{pf}

%s6 ###
\section{Exact asymptotics and uniqueness}\label{S: exact}\label{sec6}

In this section we establish the asymptotics of multiplicative
martingale functions in the class $\mathcal{T}_2(p)$ which will
quickly lead to the property of uniqueness within the same class.

For any product martingale function $\psi_p$, with speed $c_p$, where
$p\in(\underline p, \overline p]$, which belongs to the class
$\mathcal{T}_2(p)$, recall that we have defined
\[
L_p(x) = e^{(p+1)x} \bigl(1-\psi_p(x)\bigr).
\]

We start with the following first result.

\begin{theorem}\label{T: sv} Suppose that $p\in(\underline{p}, {\op
}]$ and that $\psi_p\in\mathcal{T}_2(p)$ is a product martingale
function which makes $(M(t,p,x), t \ge0)$ in (\ref{E:prod
mart})  a martingale. Then for all $\beta\ge0$
%jb changed used to be $\bet\in\R$ (also twice below)
%
\[
\lim_{x\uparrow\infty}\frac{L_p(x+\beta)}{L_p(x)} =1.
\]
That is to say, $L_p$ is additively slowly varying.
\end{theorem}

\begin{pf}%[Proof]
The proof is an adaptation of arguments found in
\cite{kyp88}. First note that the monotonicity of $L_p$ implies that
for all $\beta\ge0$
\[
\limsup_{x\uparrow\infty}\frac{L_p(x+\beta)}{L_p(x)}\geq \liminf
_{x\uparrow\infty}\frac{L_p(x+\beta)}{L_p(x)}\geq1.
\]
In turn, this implies that for each $\beta\ge0$ there exists an
increasing subsequence $\{x_k(\beta)\dvtx k\geq1\}$ tending to infinity
along which we have the following limit:
\[
\lim_{k\uparrow\infty}\frac{L_p(x_k(\beta)+\beta)}{L_p(x_k(\beta
))} =\limsup_{x\uparrow\infty}\frac{L_p(x+\beta)}{L_p(x)}\geq1.
\]

Suppose now that there exists a $\beta_0>0$ and $\eta>1$ such that
\[
\lim_{k\uparrow\infty}\frac{L_p(x_k(\beta_0) +\beta
_0)}{L_p(x_k(\beta_0))}>\eta.
\]
The monotonicity of $L_p$ implies that
for all $\beta\geq\beta_0$
\[
\liminf_{k\uparrow\infty}\frac{L_p(x_k(\beta_0) +\beta
)}{L_p(x_k(\beta_0))} = \lim_{k\uparrow\infty}\frac{L_p(x_k(\beta
_0) +\beta_0)}{L_p(x_k(\beta_0))}> \eta.
\]

The crux of the proof will be to show that this leads to a contradiction.
To this end, recall the identity (\ref{constexp}) in the proof of
Theorem \ref{stopmgs}. Starting from this expression and
with the help of a telescopic sum we have that for all $\beta\in
\mathbb{R}$
\begin{eqnarray*}
1-\psi_p(x) &=& \mathbb{E}\sum_{i}\bigl[ 1 - \psi_p\bigl(x -\log\bigl|\Pi_i\bigl(\ell
^{(p,z)}\bigr)\bigr| -c_p\ell^{(p,z)}_i\bigr)\bigr]\\
&&{}\times\prod_{j<i}\psi_p\bigl(x -\log\bigl|\Pi_j\bigl(\ell^{(p,z)}\bigr)\bigr| -c_p \ell^{(p,z)}_j\bigr).
\end{eqnarray*}
Recalling the definition of $L_p$, it follows easily that
%
%e6.1 ###
\begin{eqnarray}
\label{ktoinf}
1&=& \mathbb{E} \biggl[\sum_{i}\frac{L_p(x_k(\beta_0) -{\log}|\Pi
_i(\ell^{(p,z)})| -c_p\ell^{(p,z)}_i)}{L_p(x_k(\beta_0) )}
\bigl|\Pi_i\bigl(\ell^{(p,z)}\bigr)\bigr|^{p+1}  \nonumber
\\[-8pt]
\\[-8pt]
&&\hspace*{23pt}{}\times
e^{\Phi(p)\ell^{(p,z)}_i}\prod_{j<i}\psi_p\bigl(x_k(\beta_0) -\log\bigl|\Pi_j\bigl(\ell^{(p,z)}\bigr)\bigr|
-c_p \ell^{(p,z)}_j\bigr) \biggr].
\nonumber
\end{eqnarray}
Now pick $z\geq\beta_0$.
Next, recalling that $\psi_p(\infty)=1$ and that $-{\log}|\Pi_i(\ell
^{(p,z)}_i)| -c_p\ell^{(p,z)}_i \ge z\geq\beta_0$, we can take
limits in (\ref{ktoinf}) as $k\uparrow\infty$, applying Fatou's
Lemma twice, to reach the inequality
\[
1\geq\eta\mathbb{E}\sum_{i}\bigl|\Pi_i\bigl(\ell^{(p,z)}\bigr)\bigr|^{p+1}e^{\Phi
(p)\ell^{(p,z)}_i} .
\]
However, Theorem \ref{stopmgs} implies that the expectation above is
equal to unity and we reach a contradiction. We are forced to conclude
that $ \limsup_{x\uparrow\infty} L_p(x+\beta)/L_p(x)=1$ and the
required additive slow variation follows.
\end{pf}

The following lemma is a key ingredient which will help extract exact
asymptotics.

\begin{lemma}\label{L: Lswitch} Fix $p\in(\underline{p}, {\op}]$.
Suppose that $g\dvtx\mathbb{R}\rightarrow(0,\infty)$ is a monotone
increasing function and additive slowly varying at $+\infty$, that is
to say, it satisfies the property that for all $\beta\geq0$
\[
\lim_{x\uparrow\infty}\frac{g(x+\beta)}{g(x)} = 1.
\]
Then
\[
\lim_{z\uparrow\infty} \frac{\sum_{i} y_i^{(p)}(\ell^{(p,z)})
g(x-{\log}|\Pi_i(\ell^{(p,z)})| -c_{p}\ell^{(p,z)}_i ) }{g(x+z)W(\ell
^{(p,z)}, p)} =1,
\]
where the limit is understood almost surely when $p\in(0,\overline p]$
and in probability when $p\in(\underline p, 0]$.
\end{lemma}

\begin{pf}
The proof will follow closely ideas in Biggins and Kyprianou \cite
{BK97}, Theorem~8.6. First define
\[
\mathcal{I}_c = \bigl\{i \dvtx -\log\bigl|\Pi_i\bigl(\ell^{(p,z)}\bigr)\bigr| -c_{p}\ell
^{(p,z)}_i >z+c\bigr\}.
\]
Then, using the fact that $g$ is increasing and $ -{\log}|\Pi_i(\ell
^{(p,z)})| -c_{p}\ell^{(p,z)}_i\ge z$, we have
%
%e6.2 ###
\begin{eqnarray}
\label{upperest}
 \qquad 1&\leq&\frac{\sum_{i} y_i^{(p)}(\ell^{(p,z)}) g(x-{\log}|\Pi_i(\ell
^{(p,z)})| -c_{p}\ell^{(p,z)}_i ) }{ g(x+z) W(\ell^{(p,z)},
p)}\nonumber
\\[-8pt]
\\[-8pt]
 \qquad &\leq& \frac{g(x+z+c)}{g(x+z) }
+ \frac{
\sum_{i\in\mathcal{I}_c}
y_i^{(p)}(\ell^{(p,z)}) g(x-{\log}|\Pi_i(\ell^{(p,z)})| -c_{p}\ell
^{(p,z)}_i )
}{ g(x+ z) W(\ell^{(p,z)}, p) }.
\nonumber
\end{eqnarray}
As $g$ is additively slowly varying, we may appeal to the classical
representation of slowly varying functions (cf. Feller \cite{feller},
VIII.9) to deduce that for all $\epsilon_1, \epsilon_2 >0$ there
exists a $z_0>0$ such that for all $u>0$
\[
\sup_{z>z_0} \frac{g(z+ u)}{g(z)} \leq(1+\epsilon_1) e^{\epsilon_2 u}.
\]
This allows for the upper estimate on the second term on the right-hand
side of (\ref{upperest}) for all $z$ sufficiently large
%
%e6.3 ###
\begin{eqnarray}
\label{needLLN}
&&\frac{
\sum_{i\in\mathcal{I}_c} y_i^{(p)}(\ell^{(p,z)}) g(x-{\log}|\Pi
_i(\ell^{(p,z)})| -c_{p}\ell^{(p,z)}_i )
}
{
g(x+z) W(\ell^{(p,z)}, p)
}
\nonumber
\\[-8pt]
\\[-8pt]
&& \qquad  \leq
\frac{
\sum_{i} y_i^{(p)} (\ell^{(p,z)}) f_c(-{\log}|\Pi_i(\ell^{(p,z)})|
-c_{p}\ell^{(p,z)}_i-z)
}
{ W(\ell^{(p,z)}, p) },
\nonumber
\end{eqnarray}
where
\[
f_c(x) = \mathbf{1}_{\{x>c\}}(1+\epsilon_1)e^{\epsilon_2x}.
\]

Now note from Theorem \ref{T: SLLN} that the right-hand side of (\ref
{needLLN}) converges almost surely to $Q^{(p)}(f_c)$ when $p\in
(0,\overline p]$, where the definition of $Q^{(p)}(f_c)$ is given by
(\ref{bigpQp}). When $p\in(\underline p, 0]$, the convergence occurs
in probability and $Q^{(p)}(f_c)$ is defined by (\ref{smallpQp}). In
either case, thanks to the appropriate integrability of the function
$e^{\epsilon_2 x}$ for sufficiently small $\epsilon_2$, we have that
$Q^{(p)}(f_c)\downarrow0$ as $c\uparrow\infty$.
Moreover, in both cases, using (\ref{needLLN}) and (\ref{upperest})
we have
\begin{eqnarray*}
0&\leq&\frac{\sum_{i} y_i^{(p)}(\ell^{(p,z)}) g(x-{\log}|\Pi_i(\ell
^{(p,z)})| -c_{p}\ell^{(p,z)}_i - z) }{g(x+z)W(\ell^{(p,z)}, p)}
-1\\
&\leq&  \biggl| \frac{g(x+z+c)}{g(x+z) } -1 \biggr|
+Q^{(p)}(f_c)\\
&&{}+ \biggl|\frac{
\sum_{i} y_i^{(p)} f_c(-{\log}|\Pi_i(\ell^{(p,z)})| -c_{p}\ell^{(p,z)}_i-z)
}
{ W(\ell^{(p,z)}, p) }- Q^{(p)}(f_c) \biggr|.
\end{eqnarray*}
When $p\in(0,\overline p]$, thanks to the preceding remarks, as
$z\uparrow\infty$, the almost sure limit of the right-hand side above
can be made arbitrarily small by choosing $c$ sufficiently large. When
$p\in(\underline p, 0]$, again thanks to the preceding remarks,
we see that for each $\varepsilon>0$, we may choose $c$ sufficiently
large such that
\[
\lim_{z\uparrow\infty}\mathbb{P} \biggl( \biggl|
\frac{\sum_{i} y_i^{(p)}(\ell^{(p,z)}) g(x-{\log}|\Pi_i(\ell
^{(p,z)})| -c_{p}\ell^{(p,z)}_i - z) }{g(x+z)W(\ell^{(p,z)}, p)} -1
 \biggr|>\varepsilon \biggr) = 0,
\]
thus establishing the required convergence in probability.
\end{pf}

We are now ready to establish the asymptotics of multiplicative
martingale functions from which uniqueness will follow.
%jb moved to the beginning of the section
%For any product martingale function $\psi_p$, with speed $c_p$, where
%$p\in(\underline p, \overline p]$, which belongs to the class $
%L_p(x) = e^{(p+1)x} (1-\psi_p(x)).
%Moreover, it has already been shown in Theorem \ref{T: sv} that $L_p$
%is additively slowly varying.
It has already been shown in Theorem~\ref{T: sv} that $L_p =
e^{(p+1)x} (1-\psi_p(x))$ is additively slowly varying [with $\psi_p,
\in\mathcal{T}_2(p)$ a product martingale function with speed $c_p$].

\begin{theorem}\label{T:L cvg to cstt}
Suppose that $\psi_p$ is any product martingale function in $\mathcal
{T}_2(p)$ with speed $c_p$ [i.e., which makes $M(t,p,x)$ in (\ref{E:prod mart})  a martingale]
%jb added the above parenthesis to clarify
such that $p\in(\underline p, \overline p]$. Then there exist some
constants $k_p\in(0,\infty)$ such that
when $p\in(\underline{p}, {\op}),$ we have that
\[
L_p(x)\to k_p  \qquad \mbox{as }x\uparrow\infty
\]
and when $p={\op}$ we have
\[
\frac{L_{\overline p}(x)}{x}\to k_{\overline p} \qquad \mbox{as }x\uparrow
\infty.
\]
\end{theorem}

\begin{pf} Suppose that $p\in(\underline p, \overline p]$.
It is not difficult to show that for any given $\varepsilon>0$ we may
take $z$ sufficiently close to $1$
such that
%
%e6.4 ###
\begin{equation}
-\frac{\log z}{1-z} \in\bigl[1, (1-\varepsilon)^{-1}\bigr].
\label{estimate}
\end{equation}
Thanks to Theorem \ref{stopmgs} and, in particular, the fact that
$M(\ell^{(p,z)}, p, x)$ is a uniformly integrable martingale with
limit $M(\infty, p,x)$,
\[
-\log M(\infty,p,x) = -\lim_{z\uparrow\infty}\sum_{i}\log\psi
_p\bigl(x-\log\bigl|\Pi_i\bigl(\ell^{(p,z)}\bigr)\bigr| -c_{p}\ell^{(p,z)}_i \bigr).
\]
Since $-{\log}|\Pi_i(\ell^{(p,z)})| -c_{p}\ell^{(p,z)}_i\geq z$ and
$\psi_p(\infty)=1$, we may
apply the estimate in (\ref{estimate}) and deduce that
%
%e6.5 ###
\begin{eqnarray}
&&-e^{(p+1)x} \log M(\infty, p, x)\nonumber
\\[-8pt]
\\[-8pt] && \qquad = \lim_{z\uparrow\infty} \sum_{i}
y^{(p)}_i\bigl(\ell^{(p,z)}\bigr)L_p\bigl(x-\log\bigl|\Pi_i\bigl(\ell^{(p,z)}\bigr)\bigr| -c_{p}\ell
^{(p,z)}_i \bigr).
\label{logmglimit}
\nonumber
\end{eqnarray}

Next, we consider the restriction that $p\in(\underline{p}, {\op})$.
Recalling from Theorem \ref{stopmgs} that $W(\ell^{(p,z)},p)$
converges almost surely in $L^1$ to $W(\infty, p)$ and from Theorem
\ref{T: sv} that $L_p$ is additive slowly varying, we may
apply the conclusion of Lemma \ref{L: Lswitch} to (\ref{logmglimit})
and deduce that
%
%e6.6 ###
\begin{equation}
\frac{-e^{(p+1)x}\log M(\infty, p,x)}{W(\infty,p)} =\lim_{z\uparrow
\infty}L_p(x+z).
\label{-logM=W}
\end{equation}
The right-hand side above is purely deterministic and the left-hand
side is bounded and strictly positive which leads us to the conclusion
that there exists a constant $k_p\in(0,\infty)$ such that $L_p(x)\sim
k_p$ as $x\uparrow\infty$.

Now suppose that $p =\overline p$. From Bertoin and Rouault \cite{BeRo} (see also the method in
Kyprianou \cite{kypAIHP}) it is known that
\[
\partial W(t,\overline p,x):= \sum_{i\in\mathcal{I}(t,x)}\bigl(x-{\log}
|\Pi_i(t)| -c_{\overline p}t \bigr)y^{(\overline p)}_i(t),
\]
where
\[
\mathcal{I}(t,x) =\Bigl\{ i \dvtx \inf_{s\leq t} \{ x-{\log}|\Pi_i(t)|
-c_{\overline p}t\} >0\Bigr\}
\]
is a uniformly integrable positive martingale with mean $x$; we denote
its limit by $\partial W(\infty, \overline p, x)$.
Moreover, thanks to (\ref{left-most-speed}), it is also true that
there exists a part of the probability space, say $\gamma_x$,
satisfying $\lim_{x\uparrow\infty}\mathbb{P}(\gamma_x) =1$, such
that $\partial W(\infty, \overline p, x) = \partial W(\infty,
\overline p)$ on $\gamma_x$.\vadjust{\goodbreak}

Again, thanks to Lemma \ref{L:Markov}, we may project the limit
$\partial W(\infty, \overline p, x)$ back on to the filtration
$\mathcal{F}_{\ell^{(p,z)}}$ to obtain
\begin{eqnarray*}
\partial W\bigl(\ell^{(\overline p,z)}, \overline p, x\bigr)&: =& \sum_{i\in
\mathcal{I}(\ell^{(\overline p,z)},x)} y^{(\overline p)}_i\bigl(\ell
^{(\overline p,z)}\bigr)\bigl(x-\log\bigl|\Pi_i\bigl(\ell^{(\overline p,z)}\bigr)\bigr|
-c_{\overline p}\ell^{(\overline p,z)} \bigr)\\
 &\,\,=& \mathbb{E}\bigl(\partial
W(\infty, \overline p, x)\big|\mathcal{F}_{\ell^{(\overline p,z)}}\bigr)
\end{eqnarray*}
is a positive uniformly integrable martingale with almost sure and
$L^1$ limit $\partial W(\infty, \overline p, x)$ where
\[
\mathcal{I}\bigl(\ell^{(\overline p,z)},x\bigr) =\Bigl\{ i \dvtx \inf_{s\leq\ell
^{(\overline p,z)}} \{ x-{\log}|\Pi_i(s)| -c_{\overline p}s\} >0\Bigr\}.
\]
Note also that this implies that for each $x>0$, on $\gamma_x$ we have
%
%e6.7 ###
\begin{equation}
\lim_{z\uparrow\infty} \sum_i y^{(\overline p)}_i\bigl(\ell^{(\overline
p,z)}\bigr)\bigl(x-\log\bigl|\Pi_i\bigl(\ell^{(\overline p,z)}\bigr)\bigr| -c_{\overline p}\ell
^{(\overline p,z)} \bigr) %=\lim_{z\uparrow\infty} \sum_i (-\log|\Pi_i(
=\partial W(\infty, \overline p)
\label{ell-dW}
\end{equation}
almost surely.
%where the second equality follows on account of the fact that $W(
As we may take $x$ arbitrarily large, the above almost sure convergence
occurs on the whole of the probability space.

Next, turning to Lemma \ref{L: Lswitch}, we note that both $g(x) =
L_{\overline p}(x)$ and $g(x)=x$ are suitable functions to use within
this context. We, therefore, have for $x>0$
\[
\lim_{z\uparrow\infty}\frac{x+z}{L_{\overline p}(x+z)}\frac{\sum
_{i} y^{(\overline p)}_i(\ell^{(\overline p,z)})L_{\overline p}(x-{\log}
|\Pi_i(\ell^{(\overline p,z)})| -c_{\overline p}\ell^{(\overline
p,z)}_i )}{\sum_i y^{(\overline p)}_i(\ell^{(\overline p,z)})(x-{\log}
|\Pi_i(\ell^{(\overline p,z)})| -c_{\overline p}\ell^{(\overline
p,z)} )}=1
\]
almost surely.
Thanks to (\ref{logmglimit}) and (\ref{ell-dW}), it follows that
%
%e6.8 ###
\begin{equation}
\lim_{z\uparrow\infty}\frac{x+z}{L_{\overline p}(x+z)} =\frac
{-e^{(\overline p+1)x} \log M(\infty, \overline p, x)}{\partial
W(\infty, \overline p)}
\label{gives uniqueness}
\end{equation}
almost surely. In particular, as the left-hand side above is
deterministic and the right-hand side is a random variable in
$(0,\infty)$, it follows that the limit must be equal to some constant
in $\in(0,\infty)$ which we identify as $1/k_{\overline p}$.
The proof is complete.
\end{pf}

\begin{theorem}\label{T: unique MMF}
For $p \in(\underline{p},\op],$ there is a unique multiplicative
martingale function $\psi_p$ which is a solution to $(\ref{E:prod
mart})$ with speed $c_p$ in $\mathcal{T}_2$ (up to additive
translation in its argument). In particular, when $p\in(\underline
p, \overline p)$, the shape of the multiplicative martingale function
is given by
\[
\psi_p(x) = \mathbb{E}\bigl(e^{-e^{-(p+1)x} W(\infty, p)} \bigr)
\]
and the shape of the critical multiplicative martingale function is
given by
\[
\psi_{\overline p}(x) = \mathbb{E}\bigl(e^{-e^{-(\overline p+1)x} \partial
W(\infty, \overline p)} \bigr).
\]
\end{theorem}
\begin{pf} First, suppose that $p\in(\underline{p}, {\op})$ and
take any traveling wave $\psi_p$ at wave speed $c_p$. Thanks to the
uniform integrability of the associated\vadjust{\goodbreak} multiplicative martingale, as
well as (\ref{-logM=W}), we have that\vspace*{-1pt}
%
%e6.9 ###
\begin{equation}
\psi_p(x) = \mathbb{E}(M(\infty, p, x)) = \mathbb{E}\bigl(e^{-k_p
e^{-(p+1)x} W(\infty, p)}\bigr).
\label{psi-is-WLT}
\end{equation}
Note that from (\ref{E: integ dif equ}), if $\psi_p(x)$ is a
traveling wave, then so is $\psi_p(x+k)$ for any $k\in\mathbb{R}$.
We, therefore, deduce from (\ref{psi-is-WLT}) that traveling waves at
wave speed $c_p$ and $p\in(\underline p, \overline p)$ are unique up
to an additive translation in the argument. Moreover, without loss of
generality, the shape of the traveling wave may be taken to be of the
form given on the right-hand side of (\ref{psi-is-WLT}) but with $k_p=1$.
Exactly the same reasoning applies in the case $p=\overline p$ except
that we appeal to the distributional identity (\ref{gives uniqueness})
instead of (\ref{-logM=W}).
\end{pf}

%s7 ###
\section{\texorpdfstring{Proof of Theorem \protect\ref{T: main thm}}{Proof of Theorem 1}}
\label{S: proof of thm}\label{sec7}

Given the conclusion of Theorem \ref{T: unique MMF}, it remains only
to prove the first part of Theorem \ref{T: main thm}.
To this end, first suppose that $\psi_p \in\TT_2(p)$ and that\vspace*{-1pt}
%
%e7.1 ###
\begin{equation}
\mathcal{A}_p\psi_p \equiv0
\label{parabolic}
\end{equation}
(and hence, implicitly we understand that $\psi_p$ is in the domain of
$\mathcal{A}_p$).
Define $u(x,t):=\mathbb{E}(M( t , p,x ))$ for $x\in\mathbb{R}$ and
$t \ge0$ where $M(t,p,x)$ is given by (\ref{E:prod mart}). Also, for
convenience, write\vspace*{-0.5pt}
\[
\mathcal{L} \psi_p(x) := \int_{\nabla_1} \biggl\{\prod_{i}\psi
_p(x- \log s_i) -\psi_p(x) \biggr\}\nu(ds).
\]
The change of variable formula gives us\vspace*{-0.5pt}
\[
u(x,t) -u(x,0) = \E \biggl[ \int_0^t \frac{\partial}{\partial t}
M(s,p,x)\,ds + \sum_{s\le t} \Delta M(s,p,x)  \biggr].
\]
%
%jb The referee complains about \partial}{\partial t} M(s,p,x) ds.
Henceforth, we will use the notation
\[
z_i(s) = x- {\log}|\Pi_i(s)| -c s.
\]
Recall the Poisson point process construction of the fragmentation $X$
described in the \hyperref[intro]{Introduction}. Write $N(\cdot)$ for the Poisson random
measure on $\R_+ \times\N\times\nabla_1$ with measure intensity
$dt \otimes\# \otimes\nu(ds)$ which describes the evolution of the
fragmentation process. Using classical stochastic analysis of
semi-martingales and the Poissonian construction of fragmentation
processes, we deduce that\vspace*{-1pt}
%
%e7.2 ###
\begin{eqnarray}
\label{E: u(xt)-u(x,0)}%jb deleted the * so that a label appears, this
%equation is referenced p39l8
 && u(x,t) - u(x,0)\nonumber \\
%&=& \mathbb{E}\int_0^t \frac{\partial}{\partial t} M(\tau,p,x)d\tau
%+ \mathbb{E}  [ \sum_{s \le t}  (\prod_{i} \psi(z_i(s)) -
&& \qquad = \mathbb{E}\int_0^t \frac{\partial}{\partial t} M(\tau
,p,x)\,d\tau\nonumber\\
&& \qquad  \quad {}+\mathbb{E}\int_0^t \int_{\mathbb{N}} \int_{\nabla_1}
 \biggl\{\prod_{i\neq k} \psi_p(z_i(\tau))  \prod_{j\geq1} \psi
_p\bigl(z_k(\tau-)- \log s_j \bigr) \nonumber \\
&& \qquad  \quad
\hspace*{160pt}
  {}    -  \prod_{i} \psi_p(z_i(\tau-)) \biggr\} N(d
\tau,dk ,ds)
\\
&& \qquad = \mathbb{E} \int_0^t d\tau\cdot \prod_{i} \psi
_p(z_i(\tau)) \cdot\sum_{k} \frac{1}{\psi_p(z_k(\tau))}(-c\psi
_p'(z_k(\tau)))\nonumber\\
&&  \qquad  \quad {}  +\mathbb{E}\int_0^t d\tau\cdot\prod_{i} \psi
_p(z_i(\tau-)) \cdot
\sum_{k}
\frac{1}{\psi_p(z_k(\tau-)) } \mathcal{L}\psi_p(z_k(\tau-))\nonumber \\
&& \qquad = \mathbb{E}\int_0^t d\tau\cdot \prod_{i} \psi_p(z_i(\tau))
\cdot\sum_{k} \frac{1}{\psi_p(z_k(\tau))}
\mathcal{A}_p\psi_p(z_k(\tau))
\nonumber
\end{eqnarray}
and it is obvious that changing $\tau-$ into $\tau$ does not affect
the value of the integral in the final two steps above.

Assumption (\ref{parabolic}) now implies that for all $x \in\R$
and $t \ge0$
\[
u(x,t)= E\biggl(\prod_i \psi_p(z_i(t))\biggr) = u(x,0) =\psi_p(x).
\]
It is now a simple application of the fragmentation property to deduce that
\begin{eqnarray*} %jb added the -cs in the psi_p of the second line.
&&\mathbb{E}\biggl(\prod_{i} \psi_p\bigl(z_i(t+s)\bigr)\Big | \mathcal{F}_t \biggr)
\\
&& \qquad =\prod_{i} \mathbb{E}\biggl(\prod_{j\geq1}
\psi_p\bigl(z_i(t) - \log\bigl|\Pi^{(i)}_j(s)\bigr| -cs\bigr) \Big|\mathcal{F}_t\biggr) \\
&& \qquad =\prod_{i} \psi_p(z_i(t)),
\end{eqnarray*}
where $|\Pi^{(i)}|$ is the fragmentation process initiated by the
$i$-fragment at time $t$ of the original fragmentation process. Hence,
$\prod_i \psi_p(z_i(t))$ is a $\mathbb{F}$-martingale.

For the converse direction, we know from Theorem \ref{T: unique MMF}
that if $\psi_p \in\TT_2(p)$ makes $(M(t,p,x), t \ge0)$ a
martingale then, without loss of generality, we may take
%
%e7.3 ###
\begin{equation}
\psi_p(x) = \E\bigl(\exp\bigl(-e^{-(p+1)x} \Delta_p\bigr)\bigr),
\label{E: generic form}
\end{equation}
where $\Delta_p = W(\infty, p)$ if $p\in(\underline p, \overline p)$
and $\Delta_p = \partial W(\infty, p)$ if $p=\overline p$.
For the rest of the proof, $\psi_p$ will be given by the above
expression. Note that since $L_p(x) = e^{(p+1)x}(1-\psi_p(x))$ is
monotone increasing (see Proposition \ref{P: in T2}), we have that
\[
0 \le L_p'(y)=(p+1)L_p(y) - e^{(p+1)y} \psi_p'(y)
\]
so that
%
%e7.4 ###
\begin{equation}
\psi_p'(y) \le(p+1) \bigl(1-\psi_p(y)\bigr).
\label{E: bound psi'}
\end{equation}
This estimate and the fact that it implies the uniform boundedness of
$\psi'_p$, will be used at several points in the forthcoming text.

We start with a lemma which shows that $\mathcal{A}_p\psi_p$ is well
defined and that it is continuous.
\begin{lemma}\label{L: psi in T2 is in domain of L}
\[
[\mathcal{L}\psi_p(x)|<\infty \qquad \forall x.
\]
Furthermore, $x \mapsto\mathcal{A}_p\psi_p(x)$ is continuous.
\end{lemma}

\begin{pf}
We will use the following fact. Given $a_n$ and $b_n$ two sequences in
$[0,1]$, we have that
%
%e7.5 ###
\begin{equation} \label{E: useful bound}
\Bigl|\prod a_n -\prod b_n\Bigr| \le\sum|a_n-b_n|.
\end{equation}
As an immediate application we have that
%
%e7.6 ###
\begin{eqnarray}
\label{E: two terms}
|\mathcal{L} \psi_p(x)| &=& \Biggl|\int_{\nabla_1} \Biggl\{\prod
_{i=1}^\infty\psi_p(x- \log s_i) -\psi_p(x) \Biggr\}\nu(ds) \Biggr|
\notag\\
&\le&\int_{\nabla_1}\Biggl \{ \Biggl|\prod_{i}^\infty\psi_p(x- \log s_i)
-\psi_p(x)\Biggr| \Biggr\}\nu(ds)
\nonumber
\\[-8pt]
\\[-8pt]
& \le&\int_{\nabla_1} \{ |\psi_p(x- \log s_1) -\psi
_p(x)| \}\nu(ds)\nonumber\\
&&{} + \int_{\nabla_1} \Biggl\{ \sum_{i=2}^\infty
|\psi_p(x- \log s_i) -1| \Biggr\}\nu(ds).\notag
\end{eqnarray}
We bound the two terms separately. For the first term, chose $\epsilon
$ small enough so that $(1-\epsilon\le x \le1) \Rightarrow\log x \le
2(1-x)$ to obtain
%jb the previous line was added
%
%e7.7 ###
\begin{eqnarray}\label{E: upper bound 1}
&&\int_{\nabla_1} \bigl( \psi_p(x - \log s_1) -\psi_p(x) \bigr) \nu
(ds) \nonumber\\
&& \qquad \le\int_{\{ 1-s_1\geq\epsilon\}}\bigl(\psi_p(x - \log s_1)
-\psi_p(x)\bigr) \nu(ds)\nonumber\\
&& \qquad  \quad {} + \int_{\nabla_1 / \{1-s_1 <\epsilon\}} \bigl(\psi
_p(x - \log s_1) -\psi_p(x) \bigr) \nu(ds)\nonumber
\\[-8pt]
\\[-8pt]
&& \qquad \le\bigl(1-\psi_p(x)\bigr) \nu(\{1-s_1 \geq\epsilon\}) + (p+1)
\int_{\nabla_1 / \{1-s_1 < \epsilon\}} (- \log s_1) \nu(ds)\nonumber\\
&& \qquad \le C\bigl(1-\psi_p(x)\bigr) +(p+1) \int_{\nabla_1 / \{1-s_1 <
\epsilon\}} 2 (1-s_1) \nu(ds)\nonumber \\
&& \qquad  \le C\bigl(1-\psi_p(x)\bigr)+ C',\nonumber
\end{eqnarray}
where $C$ and $C'$ are finite constants and we have used the Mean Value
Theorem and (\ref{E: bound psi'}) in the second inequality and $\int
_{\nabla_1} (1-s_1)\nu(ds) <\infty$ in the final inequality.
%%Recalling (\ref{E: generic form}) we see that $\psi_p'(x)$ is finite
%for all $x$
This shows that $\int_{\nabla_1} ( \psi_p(x - \log s_1) -\psi_p(x)
) \nu(ds) <\infty.$

For the second term in (\ref{E: two terms}), we first observe that,
thanks to Theorem \ref{T: sv}, for each $\epsilon>0$ we can bound
%
%e7.8 ###
\begin{equation}\label{E: bound for 1 - psi}
|1-\psi_p(x)| \le c e^{-(p+1-\epsilon) x},
\end{equation}
where $c$ is a constant. Hence,
\begin{eqnarray*}
\int_{\nabla_1} \Biggl\{ \sum_{i=2}^\infty|\psi_p(x- \log s_i)
-1| \Biggr\}\nu(ds) &\le& c \int_{\nabla_1} \Biggl\{ \sum
_{i=2}^\infty e^{-(p+1-\epsilon)(x-\log s_i)}  \Biggr\}\nu(ds) \\
&\le& c e^{-(p+1-\epsilon)x} \int_{\nabla_1}\Biggl \{ \sum
_{i=2}^\infty s_i^{p+1-\epsilon}  \Biggr\}\nu(ds)
\end{eqnarray*}
as $p>\underline{p}$ we can chose $\epsilon$ small enough so that
$p-\epsilon>\underline{p}$ which then implies that
\[
\int_{\nabla_1} \Biggl\{ \sum_{i=2}^\infty s_i^{p+1-\epsilon}
\Biggr\}\nu(ds)<\infty.
\]
Hence, putting the two bounds together we see that, for each $x \in\R
,$ it holds that $|\mathcal{L}\psi_p(x)|<\infty$ and hence ,$\psi
_p$ belongs to the domain of $\mathcal{A}_p.$

Let us now show that $\mathcal{A}_p\psi_p$ is continuous. As $\psi
_p$ is $C^1(\mathbb{R})$, it is enough to show that $\LL\psi_p$ is
continuous. We start by writing
\begin{eqnarray*}
&&|\LL\psi_p(x+\epsilon) - \LL\psi_p(x) |
\\ && \qquad \le\int_{\nabla_1} \Biggl\{ \Biggl|\prod_{i}^\infty\psi_p(x
+\epsilon- \log s_i)\\
&& \qquad
\hphantom{\le\int_{\nabla_1} \Biggl\{ \Biggl|}
{} -\psi_p(x+\epsilon)-\prod_{i}^\infty\psi
_p(x- \log s_i) +\psi_p(x)  \Biggr| \Biggr\}\nu(ds).
\end{eqnarray*}
Next, we decompose the integrand as a sum
%
%e7.9 ###
\begin{eqnarray}
\label{E: fourterms}
&&\prod_{i}^\infty \psi_p(x +\epsilon- \log s_i) -\psi_p(x+\epsilon
) -\prod_{i}^\infty\psi_p(x- \log s_i) +\psi_p(x)
 \notag\\
&& \qquad = \bigl(\psi_p(x+\epsilon-\log s_1)-\psi_p(x+\epsilon)\bigr) \prod_{i\ge
2}^\infty \psi_p(x +\epsilon- \log s_i) \notag\\
&& \qquad  \quad {} + \psi_p(x+\epsilon) \Biggl( \prod_{i\ge2}^\infty \psi_p(x
+\epsilon- \log s_i) -1  \Biggr) \notag\\
&& \qquad  \quad {}  -  \bigl( \psi_p(x-\log s_1) - \psi_p(x)  \bigr) \prod
_{i\ge2}^\infty \psi_p(x - \log s_i)
\\
&& \qquad \quad {}  - \psi_p(x) \Biggl ( \prod_{i\ge2}^\infty \psi_p(x - \log
s_i) -1  \Biggr)\nonumber\\
&& \qquad=  \bigl( \psi_p(x+\epsilon-\log s_1) - \psi_p(x-\log s_1) - \psi
_p(x+\epsilon) +\psi_p(x)  \bigr)\nonumber\\
&& \qquad  \quad {}\times  \Biggl( \prod_{i\ge2}^\infty
\psi_p(x +\epsilon- \log s_i)  \Biggr) \notag\\
%jb the last factor was  ( \prod_{i\ge2}^\infty\psi_p(x - \log
%s_i) -1  )
&& \qquad \quad {} + \psi_p(x-\log s_1) \Biggl ( \prod_{i\ge2}^\infty \psi_p(x
+\epsilon- \log s_i) - \prod_{i\ge2}^\infty \psi_p(x - \log s_i)
 \Biggr)\notag\\
&& \qquad \quad {} +  \bigl( \psi_p(x+\epsilon) -\psi_p(x)  \bigr) \Biggl (
\prod_{i\ge2}^\infty \psi_p(x+\epsilon- \log s_i) -1 \Biggr).
%& + \psi_p(x)  ( \prod_{i\ge2}^\infty\psi_p(x+\epsilon-
\nonumber
\end{eqnarray}
The proof will be complete once we will have shown that the integral
with respect to $\nu(ds)$ of each term on the right-hand side of (\ref
{E: fourterms}) goes to 0 as $\epsilon\to0.$

\subsection*{First term} The first term is
\begin{eqnarray*}
%T_1(x,\epsilon)=
 &&\bigl( \psi_p(x+\epsilon-\log s_1)\\
 && \qquad   {}  - \psi_p(x-\log s_1)- \psi
_p(x+\epsilon) +\psi_p(x)  \bigr)\Biggl  ( \prod_{i\ge2}^\infty
\psi_p(x +\epsilon- \log s_i)  \Biggr).
\end{eqnarray*}
%
%jb same modif
The term $\prod_{i\ge2}^\infty \psi_p(x +\epsilon- \log s_i) $ is
uniformly bounded between 0 and 1. On the other hand,
\begin{eqnarray*}
&&  \bigl( \psi_p(x+\epsilon-\log s_1) - \psi_p(x-\log s_1) - \psi
_p(x+\epsilon) +\psi_p(x)  \bigr) \\
&& \qquad   = \bigl ( \psi_p(x+\epsilon-\log s_1) - \psi_p(x+\epsilon
) \bigr) -  \bigl( \psi_p(x-\log s_1) -\psi_p(x)  \bigr) \\
&& \qquad   \le( -\log s_1) \psi_p'\bigl(x+\eta_1(\epsilon,s_1) \bigr) + ( \log
s_1) \psi_p'\bigl(x+ \eta_2(\epsilon,s_1) \bigr),
\end{eqnarray*}
where $\eta_1(\epsilon,s_1) \in[\epsilon, \epsilon-\log s_1]$ and
$\eta_2(\epsilon,s_1) \in[0, -\log s_1].$ Observe that for each
$s_1$, because $\psi_p$ is $C^{\infty}(\mathbb{R})$ [since (\ref
{psi-is-WLT}) holds], we have that $|\eta_1(\epsilon,s_1) - \eta
_2(\epsilon,s_1)| \to0$ as $\epsilon\to0.$ (The choice of $\eta_1$
and $\eta_2$ might not be unique but by adopting the convention that
we always chose the lowest possible such value, the above argument
becomes tight.)

Fix $\delta\in(0,1)$ and decompose
\begin{eqnarray*}
&&\int_{\nabla_1} ( -\log s_1)   \bigl[\psi_p'\bigl(x+\eta_1(\epsilon
,s_1) \bigr) - \psi_p'\bigl(x+ \eta_2(\epsilon,s_1)\bigr) \bigr ] \nu(ds) \\
&& \qquad =
\int_{\{1-s_1>\delta\}} ( -\log s_1)  \bigl[\psi_p'\bigl(x+\eta
_1(\epsilon,s_1) \bigr) - \psi_p'\bigl(x+ \eta_2(\epsilon,s_1)\bigr)  \bigr] \nu
(ds) \\
&& \qquad  \quad {}   + \int_{1- s_1\leq\delta} (-\log s_1)
\bigl[\psi_p'\bigl(x+\eta_1(\epsilon,s_1) \bigr) - \psi_p'\bigl(x+ \eta_2(\epsilon
,s_1)\bigr) \bigr ] \nu(ds)
\\ && \qquad \le C(\delta) + C'(\delta)\int_{\{1- s_1\leq \delta\}}
\bigl[\psi_p'\bigl(x+\eta_1(\epsilon,s_1) \bigr)\\
&& \qquad
\hphantom{\le C(\delta) + C'(\delta)\int_{\{1- s_1\leq \delta\}}
\bigl[}
 {} - \psi_p'\bigl(x+ \eta_2(\epsilon
,s_1) \bigr) \bigr] (1-s_1)\nu(ds),
\end{eqnarray*}
where the first integral is bounded by a constant $C(\delta)$ which is
arbitrarily small according to the choice of $\delta$ since $\psi_p'$
is uniformly bounded by (\ref{E: bound psi'}). As $(1-s_1)\nu(ds)$ is
a finite measure, we can use the dominated convergence theorem and we
see that
\[
\lim_{\epsilon\to0} C'(\delta)\int_{s_1\ge1- \delta} \bigl [\psi
_p'\bigl(x+\eta_1(\epsilon,s_1) \bigr) - \psi_p'\bigl(x+ \eta_2(\epsilon,s_1)
 \bigr)\bigr] (1-s_1)\nu(ds) =0
\]
which proves that the first term converges to 0.

\subsection*{Second term} Again using (\ref{E: useful bound}) to
bound the difference of the two products in the second term we see,
 with the help of (\ref{E: bound psi'}) and the monotonicity of
$\psi_p$, that
\begin{eqnarray*}
&& \Biggl| \psi_p(x-\log s_1)  \Biggl( \prod_{i\ge2}^\infty \psi_p(x
+\epsilon- \log s_i) - \prod_{i\ge2}^\infty \psi_p(x - \log s_i)
 \Biggr) \Biggr| \\
&& \qquad  \le \sum_{i\ge2}|\psi_p(x+\epsilon-\log s_i) -\psi
_p(x-\log s_i)|\\
%&\qquad\le C \sum_{i\ge2}|\psi_p(x+\epsilon-\log s_i) -\psi_p(x-
&& \qquad \le \epsilon\sum_{i \ge2} \max\{ \psi_p'(y) \dvtx y \in
[x-\log s_i,x-\log s_i+\epsilon]\}\\
&& \qquad \le \epsilon(p+1) \sum_{i\geq2} \bigl(1 - \psi(x-\log s_i)\bigr)
\end{eqnarray*}
%
%jb I have suppressed one line and got rid of all the constants C. ok ?
%and observe that since $\psi_p$ is concave after some $x_0$ and $s_k
%we have
% C \sum_{i\ge2}|\psi_p(x+\epsilon-\log s_i) -\psi_p(x-\log s_i)| &
% &\le(p+1)C'\epsilon\sum_{i\ge2} (1-\psi_p(x-\log s_i))
%As $L_p$ is monotone increasing (see Proposition \ref{P: in T2}) we
%have that
% 0 &\le L_p'(y)=(p+1)L_p(y) - e^{(p+1)y} \psi_p'(y)
%so that
%Hence, returning to (\ref{E: return to}),
% C \sum_{i\ge2}|\psi_p(x+\epsilon-\log s_i) -\psi_p(x-\log s_i)| &
and we can use (\ref{E: bound for 1 - psi}) to see that
\[
\int_{\nabla_1} \sum_{i\ge2} \bigl(1-\psi_p(x-\log s_i)\bigr) \nu(ds)
<\infty.
\]
We conclude that the integral of the second term converges to 0 as
$\epsilon\to0.$

\subsection*{Third term} Let us now consider the third term
\[
 \bigl( \psi_p(x+\epsilon) -\psi_p(x)  \bigr)  \Biggl( \prod_{i\ge
2}^\infty \psi_p(x+\epsilon- \log s_i) -1 \Biggr).
\]
We have already shown [for the second term of (\ref{E: two terms})]
that
\[
\int_{\nabla_1}  \Biggl( \prod_{i\ge2}^\infty \psi_p(x+\epsilon-
\log s_i) -1 \Biggr) \nu(ds) <\infty
\]
and $  ( \psi_p(x+\epsilon) -\psi_p(x)  ) \to0$ as
$\epsilon\to0$, so the integral of this term also converges to 0.
\end{pf}

We now return to the proof of Theorem \ref{S: proof of thm} and show
that $\mathcal{A}_p\psi_p \equiv0$ where $\psi_p$ is as above.
Suppose for contradiction that there exists some $x$ such that
$\mathcal{A}_p\psi_p(x) >0$ [a similar argument will work to refute
the case $\mathcal{A}_p\psi_p(x) <0$].
We introduce the process $t\mapsto F (t)$ [which is a functional of the
fragmentation process $t \mapsto\Pi(t)$]
\[
F( t) := \prod_{i} \psi_p( z_i(t) ) \cdot\sum_{k} \frac{1}{\psi_p(z_k(t))}
\mathcal{A}_p\psi_p(z_k(t)).
\]
Observe that since $M$ is a martingale, we have as in (\ref{E:
u(xt)-u(x,0)}) that for all $t \ge0,   0 = \mathbb
{E}(M(t,p,x))-\mathbb{E}(M(0,p,x)) = \E\int_0^t F(s)\,ds$.

We claim that almost surely $\lim_{t \to0} F(t) = \mathcal{A}_p\psi_p(x).$
Indeed, we start by observing that as $M(x,t)=\prod_{i} \psi_p(
z_i(t) ) $ is a uniformly bounded martingale, then as $t \to0$
\[
\prod_{i} \psi_p( z_i(t) ) \to M(0,p,x)=\psi_p(x).
\]
So we only need to show that
\[
\sum_{k} \frac{1}{\psi_p(z_k(t))}
\mathcal{A}_p\psi_p(z_k(t))
\to\frac{\mathcal{A}_p\psi_p(x)}{\psi_p(x)}.
\]
However, since
\[
\frac{1}{\psi_p(z_1(t))}
\mathcal{A}_p\psi_p(z_1(t))
\to\frac{\mathcal{A}_p\psi_p(x)}{\psi_p(x)},
\]
because $\mathcal{A}_p\psi_p$ is continuous (Lemma \ref{L: psi in T2
is in domain of L}) and $z_1(t) \to x$, it is enough to show that
\[
\sum_{k \ge2}
\mathcal{A}_p\psi_p(z_k(t)) \to0
\]
[where we have used that for all $t$ and $k , 1 < 1/\psi_p(z_k(t))<
1/\psi_p(x)$].

Note that
\[
\sum_{k \ge2} \mathcal{A}_p\psi_p(z_k(t)) = -c_p \sum_{k\ge2}
\psi_p'(z_k(t)) + \sum_{k \ge2} \LL\psi_p (z_k(t)).
\]
%
%&= -c_p \E [(p+1)\sum_{k\ge2} e^{-(p+1) z_k(t) } \exp( -
%e^{-(p+1)z_k(t)} W ) ] + \sum_{k \ge2} \LL\psi_p (z_k(t))
We now turn to the sum $\sum_{k \ge2} \LL\psi_p (z_k(t)).$ Using
the bounds in (\ref{E: upper bound 1}) and (\ref{E: bound psi'}) and
the same arguments as in Lemma \ref{L: psi in T2 is in domain of L},
we see that
\begin{eqnarray*}
|\mathcal{L} \psi_p(x)| & \le&\int_{\nabla_1} \{ |\psi_p(x-
\log s_1) -\psi_p(x)| \}\nu(ds)\\
&&{} + \int_{\nabla_1} \Biggl\{
\sum_{i=2}^\infty|\psi_p(x- \log s_i) -1| \Biggr\}\nu(ds) \\
& \le& C\bigl(1-\psi_p(x)\bigr)+ C'\psi_p'(x) + C'' e^{-(p+1-\epsilon)x} \\
& \le& C\bigl(1-\psi_p(x)\bigr) + C'' e^{-(p+1-\epsilon)x},
\end{eqnarray*}
where $C,C'$ and $C''$ are (uniform in $x$) constants which may change
value from line to line.

We thus have, again appealing to (\ref{E: bound psi'}),
\[
\biggl|\sum_{k \ge2} \mathcal{A}_p\psi_p(z_k(t))\biggr| \le C \sum_{k \ge2}
\bigl(1-\psi_p(z_k(t))\bigr) + C'' e^{-(p+1-\epsilon)z_k(t)} .
\]
Using that for any $\epsilon>0$, we have $1-\psi_p(x) \le
Ce^{-(p+1-\epsilon)x}$ for some constant $C$ (which, again, may change
from line to line) we see that
\begin{eqnarray*}
\biggl|\sum_{k \ge2} \mathcal{A}_p\psi_p(z_k(t))\biggr| &\le& C \sum_{k \ge2}
|\Pi_k(t)|^{p+1-\epsilon} e^{-(p+1) (x - c_pt)} \\
& \le& C \sum_{k
\ge2} |\Pi_k(t)|^{p+1-\epsilon}
\end{eqnarray*}
and for $\epsilon$ small enough so that $p-\epsilon>\underline p$,
$\sum_{k\ge2} |\Pi_k(t)|^{p+1-\epsilon} \to0$ almost surely when
$t\to0$ on account of the fact that $W(t, p-\epsilon)\rightarrow1$
almost surely as $t\to0$. since $p >\underline{p}$. The claim that
$\lim_{t \to0} F(t) = \mathcal{A}_p\psi_p(x)$ now follows.

The almost sure right-continuity at 0 of $F$ implies that
the stopping time $\tau= \inf\{t \dvtx F(t) < \mathcal{A}_p\psi_p(x)/2\}
$ is almost surely strictly positive. Because $M$ is a uniformly
integrable martingale we must have $\E(M(\tau,p,x) )= \E(M(0,p,x))$ but
\begin{eqnarray*}
\E(M(\tau,p,x) )- \E(M(0,p,x))& =& \E\int_0^\tau F(s)\,ds \\ &\ge&
\mathcal{A}_p\psi_p(x)/2 \E(\tau) \\&>& 0,
\end{eqnarray*}
so we have a contradiction to the assumption that there exists some $x$
such that $\mathcal{A}_p\psi_p(x) >0$. This completes the proof of
Theorem \ref{S: proof of thm}.

\section*{Acknowledgments}

All three authors would like to thank   anonymous referees for their
very careful consideration of an earlier manuscript which lead to
numerous improvements.

% imsref loaded by smiklovaite, 2010-12-03 14:42:58

\printaddresses

\end{document}